\documentclass{amsart}
\usepackage[utf8]{inputenc}
\usepackage{amsmath, amssymb, amsthm, amsfonts, mathtools, array, xfrac, float, multirow, indentfirst, url, enumerate, enumitem, comment, afterpage, pifont, makecell, xcolor, hyperref, bbm}
\usepackage[margin=1.5in]{geometry}

\usepackage[backend=bibtex, sorting=none, bibstyle=alphabetic, citestyle=alphabetic, sorting=nyt,maxbibnames=99, giveninits=false, isbn=false, url=false, doi=false, eprint=false]{biblatex}
\usepackage{comment}
\usepackage{thmtools}
\usepackage{cleveref}
\allowdisplaybreaks

\newtheorem{theorem}{Theorem}
\newtheorem{lemma}[theorem]{Lemma}
\newtheorem{proposition}[theorem]{Proposition}
\newtheorem{corollary}{Corollary}
\newtheorem{remark}{Remark}

\renewcommand{\Re}{\mathrm{Re}}
\renewcommand{\Im}{\mathrm{Im}}

\theoremstyle{definition}
\newcommand{\ov}{\overline}
\newcommand{\mc}{\mathcal}

\newcommand{\md}{\,\mathrm{d}}

\newcommand{\wt}{\widetilde}

\newcommand{\eps}{\varepsilon}
\newcommand{\lcm}{\mathrm{lcm}}
\newcommand{\mf}{\mathfrak}
\newcommand{\bb}{\mathbb}

\makeatletter
\declaretheoremstyle
    [headformat={\NOTE}, 
    notebraces={}{}, 
    notefont=\bfseries, 
    preheadhook=\def\thmt@space{}, 
    numbered=no
    ]{namedtheorem}
\makeatother

\begin{document}

\author{Ghaith Hiary, Tianyu Zhao}

\title[Unconditional estimates on the argument of Dirichlet $L$-functions]{Unconditional estimates on the argument of Dirichlet $L$-functions with applications to low-lying zeros}

\address{
    GH \& TZ: Department of Mathematics, The Ohio State University, 231 West 18th
    Ave, Columbus, OH 43210, USA.
}
\email{hiary.1@osu.edu}
\email{zhao.3709@buckeyemail.osu.edu}


\begin{abstract}
    We make explicit a result of Selberg on the argument of Dirichlet $L$-functions averaged over non-principal characters modulo a prime $q$. As a corollary, we show for all sufficiently large prime $q$ that the height of the lowest non-trivial zero of the corresponding family of $L$-functions is less than $1075\cdot \frac{2\pi}{\log q}$. Here the scaling factor $\frac{2\pi}{\log q}$ is the average spacing between consecutive low-lying zeros with height at most 1, say. We also obtain a lower bound on the proportion of $L$-functions whose first zero lies within a given multiple of the average spacing. These appear to be the first explicit unconditional results of their kinds. 
\end{abstract}

\maketitle

\section{introduction}

An important statistic for low-lying zeros of a given family of $L$-functions is the height of the first (i.e., lowest) non-trivial zero. In this paper, we are concerned with estimates related to this statistic for the family of Dirichlet $L$-functions to a prime modulus. In general, low-lying zeros of Dirichlet $L$-functions have been extensively studied in the literature; a non-exhaustive list of works in recent decades includes \cite{Mur90,BalMur92,OS99,IwaSar99,Sound00,HuRu03,Bui12,KhaNgo16,Pratt19,CCM22,KMN22,DPR23,QinWu25}. Some of these works deal specifically with the proportion of non-vanishing at the central point, which is another interesting problem but will not be the addressed in the current paper. 

Throughout let $q$ denote a large prime and $\chi$ a non-principal Dirichlet character modulo $q$. Recall that the $L$-function associated to $\chi$ is defined by
\[
L(s,\chi)=\sum_{n=1}^\infty \frac{\chi(n)}{n^s}, \quad \Re(s)>1,
\]
which admits an analytic continuation to $\bb{C}$. To study the distribution of non-trivial zeros $\rho=\beta+i\gamma$ of $L(s,\chi)$, we work with the argument
\[
S(t,\chi)=-\frac{1}{\pi}\int_{1/2}^\infty \Im \frac{L'}{L}(\sigma+it,\chi)\md \sigma.
\]
Here it is assumed that $t\neq \gamma$ for any $\gamma$, otherwise we set 
\[
S(t,\chi)=\lim_{\eps\to 0}\frac{1}{2}(S(t+\eps,\chi)+S(t-\eps,\chi)).
\]
If we write
\[
\wt{S}(t,\chi)=S(t,\chi)+S(t,\ov{\chi}), \quad t>0,
\]
then according to the Riemann\textendash von Mangoldt formula the number of non-trivial zeros with $-t\leq \gamma\leq t$ can be expressed as
\begin{equation}\label{Riemann-von mangoldt}
    N(t,\chi)=\frac{t}{\pi}\log\frac{qt}{2\pi e}+\wt{S}(t,\chi)+O(1)
\end{equation}
where the error term is fully understood. The irregularities in the counting function $N(t,\chi)$ therefore completely come from the term $\wt{S}(t,\chi)$. By analogy with the classical results for the Riemann zeta-function, we have 
\begin{equation}\label{S(t,chi) individual bound}
    S(t,\chi)=O(\log q(t+3))
\end{equation}
where the implied constant is uniform in $q$ and $t>0$, and this can be improved to
\[
S(t,\chi)=O\left(\frac{\log q(t+1)}{\log\log q(t+3)}\right)
\]
conditionally on the generalized Riemann hypothesis (GRH), which asserts that each non-trivial zero has $\beta=1/2$. For fixed $t$, while the previous two upper bounds do not prevent individual $S(t,\chi)$ to get arbitrarily large as $q\to \infty$, the average over $\chi\bmod q$ in fact remains bounded. This was first shown by Titchmarsh \cite{Titch31} under GRH, and quite remarkably Selberg \cite[Theorem 8]{Sel46} managed to establish this unconditionally. Writing
\[
\bb{E}[S(t,\chi)]=\frac{1}{q-2}\sum_{\chi\neq \chi_0}S(t,\chi)
\]
where $\chi$ ranges over all $q-2$ non-principal characters modulo $q$, Selberg proved for any $0< \eps\leq 1/4$ and $|t|\leq q^{1/4-\eps}$ that
\begin{align}\label{Selberg bound}
    \bb{E}[S(t,\chi)]=O(1)
\end{align}
where the implied constant depends only on $\eps$, not $q$. 

We now return to the discussion on low-lying zeros. In view of the zero-counting formula \eqref{Riemann-von mangoldt}, Selberg's bound \eqref{Selberg bound} and the observation $\bb{E}[\wt{S}(t,\chi)]=2\bb{E}[S(t,\chi)]$, we have for any fixed $t>0$
\begin{equation}\label{average spacing}
    \frac{1}{q-2}\sum_{\chi\neq \chi_0} N(t,\chi)\sim \frac{t}{\pi}\log q
\end{equation}
as $q\to \infty$. \footnote{Note that without averaging over $\chi$, \eqref{average spacing} is not necessarily true unconditionally as the error term $\wt{S}(t,\chi)$ could be of the same order of magnitude as the main term, in view of \eqref{S(t,chi) individual bound}.} In this sense,
the average spacing between the ordinates of consecutive low-lying zeros of $L(s,\chi)$ is $\frac{2\pi}{\log q}$, and consequently it is natural to normalize by this factor when measuring the height of the first zero. By studying the one-level density of the zeros of the family $L(s,\chi)$ modulo $q$, Hughes and Rudnick \cite[Corollary 8.2]{HuRu03} showed under GRH that
\begin{equation}\label{Hughes Rudnick first zero}
    \min_{\chi\neq \chi_0}\frac{|\gamma_{\chi,0}|\log q}{2\pi} \leq \frac{1}{4}+o(1), \quad q\to \infty
\end{equation}
where for each $\chi$
\[
|\gamma_{\chi,0}|:=\min_{\gamma_\chi}|\gamma_\chi|.
\]
\footnote{With additional averaging over moduli, Drappeau, Pratt and Radziwiłł \cite{DPR23} extended the support in Hughes and Rudnick's one-level density result, and as a consequence, after combining \cite[Theorem 1]{DPR23} and \cite[Theorem 8.1]{HuRu03} one obtains
\[
\min_{Q\leq q\leq 2Q}\min_{\chi\neq \chi_0}\frac{|\gamma_{\chi,0}|\log Q}{2\pi} \leq \frac{1}{2(2+\frac{50}{1093})}+o(1), \quad Q\to \infty.
\]} 
In other words, the lowest zero of this family has at most 1/4 times the expected height. Recently, motivated by the connection between \eqref{Riemann-von mangoldt} and \eqref{Selberg bound}, the second author
\cite{Zhao2} gave another proof of \eqref{Hughes Rudnick first zero} by showing that 
\[
\bb{E}[S(t,\chi)]\leq \frac{1}{4}+o(1)
\]
for $|t|\leq 1$, say. However, no explicit unconditional analogue of \eqref{Hughes Rudnick first zero} exists in the literature. Hence, our main goal in this paper is to make explicit Selberg's estimate \eqref{Selberg bound} in order to furnish an unconditional upper bound on the first zero. It is worth remarking that our results are not completely explicit, though, namely, when we state ``for $q\geq q_0$" we do not compute a specific value of $q_0$ beyond which the estimate holds true.

\begin{theorem}\label{main theorem}
    For $|t|\le 1$ and $q\geq q_0$,
    \begin{equation*}
        |\bb{E}[S(t,\chi)]|< 
        C_0:=1075.
    \end{equation*}
\end{theorem}

\begin{corollary}\label{corollary 1}
    For $q\geq q_0$,
    \[
    \min_{\chi\neq \chi_0}\frac{|\gamma_{\chi,0}|\log q}{2\pi} < C_0.
    \]
\end{corollary}

In proving Theorem~\ref{main theorem}, we obtain an explicit version of a zero-density estimate of Selberg \cite[Theorem 4]{Sel46} that might be of independent interest. This result is particularly effective near the half-line. Moreover, unlike most zero-density results for Dirichlet $L$-functions in the literature, here we allow $t$ to be small.

\begin{theorem}\label{theorem zero density}
    Fix $0<\eps\leq 1/4$ and $0<\kappa<\eps/2$. Then for all $q\geq q_0(\kappa)$, $\sigma\geq \frac{1}{2}+\frac{5}{8\kappa\log q}$, and $|t_1|,|t_2|\leq q^{1/4-\eps}$ with $t_2-t_1\geq \frac{1.73}{\kappa \log q}$, we have
    \[
     \sum_{\chi\neq \chi_0} N(\sigma; t_1,t_2; \chi)< \left(4.79\kappa+\frac{4.12}{2(t_2-t_1)\log q-1.73/\kappa}\right) q^{1-2\kappa(\sigma-1/2)}(t_2-t_1)\log q,
    \]
    where $N(\sigma; t_1,t_2; \chi)$ denotes the number of zeros of $L(s,\chi)$ with $\beta\geq \sigma$ and $t_1\leq \gamma\leq t_2$.  
\end{theorem}

In addition to \eqref{Hughes Rudnick first zero}, Hughes and Rudnick \cite[Corollary 8.3]{HuRu03} gave a conditional lower bound on the proportion of $L$-functions having small first zeros by computing higher moments of the one-level density. Their result was improved by the second author in \cite[Theorem 7]{Zhao1}, and subsequently also in \cite[Corollary 2]{Zhao2}. In particular, the idea in the latter work is to estimate the mean square $\bb{E}[\wt{S}(t,\chi)^2]$ under GRH and exploit \eqref{Riemann-von mangoldt} again. Following the same strategy, we are able to provide an unconditional lower bound for the proportion being considered.

\begin{theorem}\label{theorem: mean square S(t,chi)}
    Fix $\beta>0$. For all $q\geq q_0(\beta)$ and $|t|\leq \frac{2\pi\beta}{\log q}$,
    \[
    \bb{E}[\wt{S}(t,\chi)^2] <
    \left(2C_0+\frac{\sqrt{2}}{\pi}\sqrt{\int_0^{\frac{3\beta}{50}}\frac{\sin(2\pi y)^2}{y}\md y}\right)^2
    \]
    where $C_0=1075$ is the constant from Theorem~\ref{main theorem}. Furthermore, for all $\beta\geq 0$ and $|t|\leq \frac{2\pi\beta}{\log q}$,
    \begin{equation}\label{E[S(t,chi)^2] asymptotics}
        \limsup_{q\to \infty} \bb{E}[\wt{S}(t,\chi)^2] \leq \frac{\log (\beta+3)}{\pi^2}+O\left(\sqrt{\log(\beta+3)}\right)
    \end{equation}
    where the implied constant is absolute.
\end{theorem}

\begin{corollary}\label{corollary 2}
    For $\beta>C_0$,
    \[
    \liminf_{\substack{q\to \infty}} \frac{1}{q-2}\#\bigg\{\chi\neq \chi_0: \frac{|\gamma_{\chi,0}|\log q}{2\pi}<\beta\bigg\}\\
    >
    \dfrac{(2\beta-2C_0)^2}{\displaystyle 4\beta^2-8C_0\beta+\left(2C_0+\frac{\sqrt{2}}{\pi}\sqrt{\int_0^{\frac{3\beta}{50}}\frac{\sin(2\pi y)^2}{y}\md y}\right)^2}.
    \]
\end{corollary}

In particular, for any $\beta>C_0$ a positive proportion of Dirichlet $L$-functions have the height of their first zero below $\beta$ times the average spacing, and due to \eqref{E[S(t,chi)^2] asymptotics} this proportion tends to 1 as $\beta\to \infty$ at a rate of at least $1-O(\frac{\log \beta}{\beta^2})$.

\subsection{Notations}
Throughout $s=\sigma+it$ denotes a complex number, $\rho=\beta+i\gamma$ denotes a non-trivial zero of $L(s,\chi)$, $C$ and $q_0$ are positive constants and may not mean the same at each occurrence. As usual, $\mu(n)$ and $\Lambda(n)$ denote the M\"{o}bius function and von Mangoldt's function, respectively.

\section{Proofs of Corollaries ~\ref{corollary 1} and ~\ref{corollary 2}}

We first prove the two corollaries on low-lying zeros of $L(s,\chi)$. They follow from Theorems~\ref{main theorem} and ~\ref{theorem: mean square S(t,chi)} in straightforward ways, along the same lines as the proofs of \cite[Corollaries 1 and 2]{Zhao2}.

By the argument principle and the functional equation for $L(s,\chi)$, 
\begin{equation}\label{N(t,chi) for small t}
    \begin{split}
    N(t,\chi)=&\frac{t}{\pi}\log \frac{q}{\pi} + \wt{S}(t,\chi) +       \frac{1}{2\pi} \int_{-t}^t \frac{\Gamma'}{\Gamma}\left(\frac{1}{4}+\frac{\delta_\chi}{2}+i\frac{u}{2}\right)\md u\\
    =&\frac{t}{\pi}\log \frac{q}{\pi} + \wt{S}(t,\chi) +o(1), \quad t\to 0.
    \end{split}
\end{equation}
where $\Gamma$ is the usual Gamma function. This formula is more precise than \eqref{Riemann-von mangoldt} as $t\to 0$. Averaging both sides over non-principal $\chi\bmod q$ and applying the lower bound $\bb{E}[\tilde{S}(t,\chi)]>-2C_0$ provided by Theorem~\ref{main theorem}, we see that $\sum_{\chi\neq \chi_0} N(t,\chi)>0$ when $q\geq q_0$ and $\frac{t}{2\pi}\log \frac{q}{\pi}>C_0$. In fact the last statement is true when $\frac{t}{2\pi}\log \frac{q}{\pi}>C_0-\delta$ for a tiny positive number $\delta$ since we left a little room to spare in the statement of Theorem ~\ref{main theorem}. This yields Corollary~\ref{corollary 1}. 

Now fix $\beta\geq C_0$ and consider the subset of characters
\[
\mc{Q}_\beta:=\left\{\chi \neq \chi_0:  \frac{|\gamma_{\chi,0}|\log q}{2\pi}<\beta \right\}.
\]
Setting $t=\frac{2\pi\beta}{\log q}$ in \eqref{N(t,chi) for small t}, we have
\[
\mc{Q}_\beta=\left\{\chi\neq \chi_0: \wt{S}\left(\frac{2\pi\beta}{\log q},\chi\right)+2\beta+o(1)>0\right\}
\]
where the $o(1)$ term tends to 0 as $q\to \infty$. Using the Cauchy\textendash Schwarz inequality and completing the sum over characters using non-negativity, we bound
\begin{align*}
    \# \mc{Q}_\beta \geq \frac{\left[\sum_{\chi\neq \chi_0}\left(\wt{S}(\frac{2\pi\beta}{\log q},\chi)+2\beta+o(1)\right)\right]^2}{\sum_{\chi\neq \chi_0}\left(\wt{S}(\frac{2\pi\beta}{\log q},\chi)+2\beta+o(1))\right)^2},
\end{align*}
and consequently
\begin{align*}
    \liminf_{q\to \infty} \frac{\#\mc{Q}_\beta}{q-2}\geq & \liminf_{q\to \infty}  \frac{(2\beta+\bb{E}[\wt{S}(\frac{2\pi\beta}{\log q},\chi)])^2}{4\beta^2+4\beta\cdot \bb{E}[\wt{S}(\frac{2\pi\beta}{\log q},\chi)]+\bb{E}[\wt{S}(\frac{2\pi\beta}{\log q},\chi)^2]}\\
    \geq & \liminf_{q\to \infty} \frac{(2\beta-2C_0)^2}{4\beta^2-8C_0\beta+\bb{E}[\wt{S}(\frac{2\pi\beta}{\log q},\chi)^2]}.
\end{align*}
Now we readily obtain Corollary~\ref{corollary 2} from Theorem~\ref{theorem: mean square S(t,chi)}.

\section{Proof of the zero-density Theorem~\ref{theorem zero density}}

The goal of this section is to establish Theorem~\ref{theorem zero density}, a key ingredient in the proof of Theorem~\ref{main theorem}. In preparation, we refine and make explicit a number of results in \cite[\S{4}]{Sel46}. 
Roughly speaking, the strategy is to bound the number of zeros of $L(s,\chi)$ by that of $L\cdot \psi$ where $\psi$ is a mollifier constructed such that $L\cdot \psi\approx 1$ to the right of the half-line when averaged over $\chi\bmod q$. By a Littlewood-type Lemma~\ref{lemma littlewood}, it then suffices to control $\log|L\cdot \psi|$, which is small on average.

\subsection{Some sums involving the M\"{o}bius function}
\begin{lemma}[c.f. {\cite[Lemma 11]{Sel46}}] \label{lemma mu(n)}
    Let $x>1$ and $r$ be a positive integer. For $\ell\in\{1,2\}$, define
    \[
    M_\ell(r,x):=\sum_{\substack{n<x \\ (r,n)=1}}\frac{\mu(n)}{n}\left(\log \frac{x}{n}\right)^\ell.
    \]
    Then
    \begin{align*}
        M_\ell(r,x)=&\ell! \Bigg(\log x-\gamma_0-\sum_{p|r}\frac{\log p}{p-1}\Bigg)^{\ell-1}\prod_{p|r}\frac{1}{1-p^{-1}}+O\Bigg(e^{-C\sqrt{\log x}} \prod_{p|r}\frac{1}{1-p^{-3/4}}\Bigg)
    \end{align*}
    where $\gamma_0=0.577\ldots$ is Euler's constant.
\end{lemma}

\begin{remark}
    In comparison to \cite[Lemma 11]{Sel46}, here we obtain an asymptotic expression (instead of just an upper bound) for $M_\ell(r,x)$, which is critical to establishing several asymptotic formulas for sums involving $\mu(n)$ in the next few lemmas. This is achieved by a more careful choice of contour when evaluating a complex integral.
\end{remark}

\begin{proof}
    From the identities
    \[
    \frac{\ell!}{2\pi i}\int_{1-i\infty}^{1+i\infty}\frac{u^s}{s^{\ell+1}}\md s=
    \begin{cases}
        0 & 0\leq u\leq 1,\\
        (\log u)^\ell & u\geq 1
    \end{cases}
    \]
    and
    \[
    \sum_{\substack{n=1\\ (r,n)=1}}^\infty \frac{\mu(n)}{n^{1+s}}=\frac{1}{\zeta(1+s)}\prod_{p|r}\frac{1}{1-p^{-1-s}},\quad \sigma>0,
    \]
    we arrive at the Perron-type formula
    \begin{align*}
        \sum_{\substack{n<x \\ (r,n)=1}}\frac{\mu(n)}{n}\left(\log \frac{x}{n}\right)^\ell=&\frac{\ell!}{2\pi i}\int_{1-i\infty}^{1+i\infty}\frac{x^s}{s^{\ell+1}\zeta(1+s)}\prod_{p|r}\frac{1}{1-p^{-1-s}}\md s.
    \end{align*}
    We deform the line of integration to $\Gamma$, the boundary of a zero-free region of $\zeta(s)$ (the classical zero-free region suffices for our purposes) shifted to the left by 1, such that $\Gamma$ stays to the right of the pole of $\prod_{p|r}(1-p^{-1-s})^{-1}$ at $s=-1$.
    For example, according to \cite{MTY24} one can take $\Gamma=\{s: \sigma=-C_1/\log(|t|+3), \: t\in \bb{R}\}$ where $C_1=1/5.56$.
    The residue theorem then gives
    \begin{align*}
         M_\ell(r,x)=&R_\ell+\tilde{R}_\ell
    \end{align*}
    where
    \[
    R_1:\prod_{p|r}\frac{1}{1-p^{-1}},\quad R_2:=2\Bigg(\log x-\gamma_0-\sum_{p|r}\frac{\log p}{p-1}\Bigg)\prod_{p|r}\frac{1}{1-p^{-1}},
    \]
    \[
    \tilde{R}_\ell:=\frac{\ell!}{2\pi i}\int_{\Gamma}\frac{x^s}{s^{\ell+1}\zeta(1+s)}\prod_{p|r}\frac{1}{1-p^{-1-s}}\md s.
    \]
    \footnote{The constant $\gamma_0$ appears in $R_2$ because of the asymptotic expansion $\zeta(1+s)=s^{-1}+\gamma_0+O(|s|)$ as $s\to 0$.} By \cite[Theorem 3.11]{Titch}, $\frac{1}{\zeta(1+s)}\ll \log (|t|+3)$ for $s\in \Gamma$. Also $\sigma>-1/4$ for $s\in \Gamma$, and hence
    \begin{align*}
        |\tilde{R}_\ell|\ll \int_0^\infty \frac{x^{-C_1/\log(|t|+3)}\log(|t|+3)}{(|t|+3)^{\ell+1}} \md t \cdot \prod_{p|r}\frac{1}{1-p^{-3/4}}.
    \end{align*}
    To estimate the integral, we split it into two parts so that 
    \[
    \int_0^\Delta+\int_\Delta^\infty \ll \Delta x^{-C_1/\log \Delta}+\frac{\log \Delta}{\Delta^\ell}.
    \]
    Choosing $\Delta= e^{C_2\sqrt{\log x}}$ with $C_2=\sqrt{C_1/2}$, we see that the above is $\ll e^{-C_3\sqrt{\log x}}$ for any $C_3<C_2$. This proves the lemma. 
\end{proof}

From now on $\xi$ denotes a real number that is $\geq 2$. Define
\begin{equation}\label{def: lambda_n}
    \lambda_n=\lambda_n(\xi):=
    \begin{cases}
        \mu(n) & 1\leq n\leq \xi,\\
        \mu(n)\frac{\log (\xi^2/n)}{\log \xi} & \xi\leq n\leq \xi^2. 
    \end{cases}
\end{equation}

\begin{lemma}[c.f. {\cite[Lemma 12]{Sel46}}]\label{lemma lambda_n}
    For some appropriate constant $C>0$ and $\ell\in\{1,2\}$, put
    \[
    \mc{E}_\ell(r):=e^{-C\sqrt{\log(\xi^\ell/r)}}  \prod_{p|r} \frac{1}{1-p^{-3/4}}.
    \]
    Then
    \begin{equation}\label{lambda_n/n}
        \sum_{\substack{n<\xi^2\\ r|n}} \frac{\lambda_n}{n}=\frac{\mu(r)}{r\log \xi}\cdot
        \begin{cases}
            O(\mc{E}_1(r)) & 1\leq r\leq \xi,\\
            \displaystyle
            \prod_{p|r} \frac{1}{1-p^{-1}} +  O(\mc{E}_2(r)) & \xi\leq r<\xi^2,
        \end{cases}
         \\
    \end{equation}
    
    \begin{equation}\label{lambda_n/n log n}
        \sum_{\substack{n<\xi^2\\ r|n}} \frac{\lambda_n}{n}\log n=\frac{\mu(r)}{r}\cdot 
        \begin{cases}
            \displaystyle
            -\prod_{p|r} \frac{1}{1-p^{-1}} +  O(\mc{E}_1(r)) & 1\leq r\leq \xi,\\
            \displaystyle
            \frac{2}{\log \xi}\Bigg(\log\frac{r}{\xi}+\gamma_0+\sum_{p|r}\frac{\log p}{p-1}\Bigg)\prod_{p|r} \frac{1}{1-p^{-1}} +  O(\mc{E}_2(r)) & \xi\leq r<\xi^2.
        \end{cases}
    \end{equation}

\end{lemma}

\begin{proof}
    It follows from the definition \eqref{def: lambda_n} that
    \begin{align*}
        \sum_{\substack{n<\xi^2\\ r|n}} \frac{\lambda_n}{n}=&\frac{1}{\log \xi}\sum_{\substack{n<\xi^2\\r|n}}\frac{\mu(n)}{n}\log \frac{\xi^2}{n}-\frac{1}{\log \xi}\sum_{\substack{n<\xi \\r|n}}\frac{\mu(n)}{n}\log \frac{\xi}{n}\\
        =&\frac{\mu(r)}{r\log \xi}\Bigg(\sum_{\substack{m<\xi^2/r \\(r,m)=1}}\frac{\mu(m)}{m}\log \frac{\xi^2}{rm}-\sum_{\substack{m<\xi/r \\(r,m)=1}}\frac{\mu(m)}{m}\log \frac{\xi}{rm}\Bigg)\\
        =& \frac{\mu(r)}{r\log \xi}\left(M_1(r,\xi^2/r)-M_1(r,\xi/r)\right).
    \end{align*}
    Here $M_1(r,\xi/r)=0$ if $\xi\leq r$. Applying Lemma~\ref{lemma mu(n)} with $\ell=1$ yields \eqref{lambda_n/n}. To see \eqref{lambda_n/n log n}, write
    \[
    \sum_{\substack{n<\xi^2\\ r|n}} \frac{\lambda_n}{n} \log n = \log r  \sum_{\substack{n<\xi^2\\ r|n}} \frac{\lambda_n}{n}+ \sum_{\substack{n<\xi^2\\ r|n}} \frac{\lambda_n}{n} \log \frac{n}{r}.
    \]
    The first term is handled by \eqref{lambda_n/n}. Furthermore, we find by comparing coefficients of $\mu(n)$ that
    \begin{align*}
        \sum_{\substack{n<\xi^2\\ r|n}} \frac{\lambda_n}{n} \log \frac{n}{r} = \frac{\mu(r)}{r\log \xi}\left(\log \frac{\xi^2}{r} M_1(r,\xi^2/r)-\log \frac{\xi}{r}  M_1(r,\xi/r)+M_2(r,\xi/r)-M_2(r,\xi^2/r)\right).
    \end{align*}
    Another application of Lemma~\ref{lemma mu(n)} finishes the proof.
\end{proof}

In the following lemma we provide asymptotic formulas for several double sums involving $\lambda_n$ and the greatest common divisor. In particular, they refine the corresponding estimates that appear in the proof of \cite[Lemma 13]{Sel46}.

\begin{lemma}\label{lemma lambda_n gcd}
    \begin{equation}\label{first double sum lambda_n}                    
        \sum_{n_i<\xi^2}\frac{\lambda_{n_1}\lambda_{n_2}}{n_1n_2}\gcd (n_1,n_2) = \frac{1}{\log \xi}+O\left(\frac{1}{(\log\xi)^2}\right).
    \end{equation}

    \begin{equation}\label{second double sum lambda_n}               
        \sum_{n_i<\xi^2}\frac{\lambda_{n_1}\lambda_{n_2}}{n_1n_2}\gcd(n_1,n_2)\log n_1 = 1+O\left(\frac{\log\log\xi}{\log\xi}\right)
    \end{equation}

    \begin{equation}\label{third double sum lambda_n}       
        \sum_{n_i<\xi^2}\frac{\lambda_{n_1}\lambda_{n_2}}{n_1n_2}\gcd(n_1,n_2)\log(\gcd(n_1,n_2))= \frac{3}{2}+O\left(\frac{\log\log \xi}{\log\xi}\right).
    \end{equation}
\end{lemma}

\begin{proof}
    To derive the first assertion, we start by expressing the sum as
    \[
    \sum_{n_i<\xi^2}\frac{\lambda_{n_1}\lambda_{n_2}}{n_1n_2}\gcd (n_1,n_2)=\sum_{n_i<\xi^2}\frac{\lambda_{n_1}\lambda_{n_2}}{n_1n_2}\sum_{r|\gcd(n_1,n_2)}\varphi(r)=\sum_{r<\xi^2}\varphi(r) \bigg(\sum_{\substack{n<\xi^2\\r|n}}\frac{\lambda_n}{n}\bigg)^2
    \]
    where $\varphi(r)$ denotes Euler's totient function. Writing $\varphi(r)=r\prod_{p|r}(1-p^{-1})$ and applying Lemma~\ref{lemma lambda_n},
    we obtain
    \begin{multline*}
        \sum_{n_i<\xi^2}\frac{\lambda_{n_1}\lambda_{n_2}}{n_1n_2}\gcd (n_1,n_2)= \frac{1}{(\log \xi)^2}\Bigg\{\sum_{\xi\leq r<\xi^2} \frac{\mu(r)^2}{\varphi(r)}+O\Bigg(\sum_{r<\xi} \frac{1}{r}e^{-2C\sqrt{\log (\xi/r)}}\sum_{d|r}\frac{1}{\sqrt{d}}\Bigg)\\
        +O\Bigg(\sum_{\xi\leq r<\xi^2}\frac{1}{r}e^{-C\sqrt{\log (\xi^2/r)}}\sum_{d|r} \frac{1}{\sqrt{d}}\Bigg)\Bigg\}
    \end{multline*}
    where for the big-$O$ error terms we used the bound
    \[
    \prod_{p|r}\frac{1-p^{-1}}{(1-p^{-3/4})^2}\ll \prod_{p|r} \left(1+\frac{1}{\sqrt{p}}\right)=\sum_{d|r} \frac{1}{\sqrt{d}}.
    \]
    The sum in the first error term can be rewritten as
    \[
    \sum_{d<\xi}\frac{1}{d\sqrt{d}}\sum_{d'\leq \xi/d}\frac{1}{d'}e^{-2C\sqrt{\log(\xi/(dd'))}}.
    \]
    By integral comparison the inner sum over $d'$ is $O(1)$ uniformly for all $d$, so the entire sum is $O(1)$ as well. A similar estimate holds for the second error term. As for the main term, recall the asymptotic formula (see \cite[Lemma 2.3]{Sit85})
    \begin{equation}\label{totient sum}
        \sum_{r\leq x} \frac{\mu(r)^2}{\varphi(r)}=\log x+\gamma_0+\sum_p \frac{\log p}{p(p-1)}+O\left(\frac{1}{\sqrt{x}}\right),
    \end{equation}
    from which \eqref{first double sum lambda_n} follows immediately.

    Next, similarly to the above argument, we find by using Lemma~\ref{lemma lambda_n} that 
    \begin{align*}
        \sum_{n_i<\xi^2}\frac{\lambda_{n_1}\lambda_{n_2}}{n_1n_2}&\gcd(n_1,n_2)\log n_1\\
        =&\sum_{r<\xi^2}\varphi(r) \bigg(\sum_{\substack{n<\xi^2\\r|n}}\frac{\lambda_n}{n}\bigg)\bigg(\sum_{\substack{n<\xi^2\\r|n}}\frac{\lambda_n}{n}\log n\bigg)\\
        =&\frac{2}{(\log \xi)^2}\sum_{\xi\leq r<\xi^2}\frac{\mu(r)^2}{\varphi(r)}\Bigg(\log\frac{r}{\xi}+\gamma_0+\sum_{p|r}\frac{\log p}{p-1}\Bigg)+O\left(\frac{1}{\log\xi}\right)\\
        =&\frac{2}{(\log \xi)^2}\sum_{\xi\leq r<\xi^2}\frac{\mu(r)^2}{\varphi(r)}\left(\log\frac{r}{\xi}+O(\log\log r)\right)+O\left(\frac{1}{\log\xi}\right).
    \end{align*}
    By partial summation and \eqref{totient sum},
    \begin{align*}
        \sum_{\xi\leq r<\xi^2}\frac{\mu(r)^2}{\varphi(r)}\log r=&\log \xi^2 \sum_{r<\xi^2}\frac{\mu(r)^2}{\varphi(r)}-\log \xi \sum_{r< \xi}\frac{\mu(r)^2}{\varphi(r)}-\sum_{\xi\leq r<\xi^2} \Bigg(\sum_{\xi\leq n\leq r}\frac{\mu(n)^2}{\varphi(n)}\Bigg)\log \frac{r+1}{r}\\
        =&4(\log \xi)^2-(\log \xi)^2-\sum_{\xi\leq r<\xi^2} (\log r+O(1))\left(\frac{1}{r}+O\left(\frac{1}{r^2}\right)\right)\\
        =& \frac{3}{2}(\log \xi)^2+O(\log \xi).
    \end{align*}
    The other terms can be treated directly by using \eqref{totient sum}, and this gives \eqref{second double sum lambda_n}.
    
    It remains to show \eqref{third double sum lambda_n}. Put
    \[
    \varphi'(r):=r\sum_{d|r}\frac{\mu(d)}{d}\log\frac{r}{d}.
    \]
    We claim that
    \begin{equation}\label{modified totient function}
        \varphi'(r)=\varphi(r)\bigg(\log r+\sum_{p|r}\frac{\log p}{p-1}\bigg).
    \end{equation}
    Note that $r\sum_{d|r}\mu(d)/d=r\prod_{p|r}(1-1/p)=\varphi(r)$, so it suffices to establish the following identity:
    \[
    -r\sum_{d|r}\frac{\mu(d)}{d}\log d=\varphi(r) \sum_{p|r}\frac{\log p}{p-1}.
    \]
    Plainly we may assume that $r$ is squarefree. The left-hand side can be rewritten as
    \[
    -r\sum_{d|r}\frac{\mu(d)}{d}\sum_{d'|d}\Lambda(d')=-r\sum_{d'|r}\Lambda(d')\sum_{\substack{d'|d\\d|r}}\frac{\mu(d)}{d}=-r\sum_{d'|r}\Lambda(d')\sum_{k|\frac{r}{d'}}\frac{\mu(kd')}{kd'}.
    \]
    In the last sum we have $\mu(kd')=\mu(k)\mu(d')=-\mu(k)$ since $\gcd(k,d')=1$ by the assumption that $r$ is squarefree and $d'$ contributes to the sum only when it is a prime. Thus, the above equals
    \[
    \sum_{d'|r}\Lambda(d')\sum_{k|\frac{r}{d'}}\mu(k)\frac{r/d'}{k}=\sum_{d'|r}\Lambda(d')\varphi(r/d')=\varphi(r) \sum_{p|r}\frac{\log p}{p-1},
    \]
    where the last equality can be found in \cite[Lemma 3.2]{FM13}. This proves \eqref{modified totient function}. We also have
    \[
    \sum_{r|n}\varphi'(r)=\sum_{r|n} r\sum_{d|r}\frac{\mu(d)}{d}\log \frac{r}{d}=\sum_{d|n}\mu(d)\sum_{k|\frac{n}{d}}k\log k,
    \]
    which implies that
    \[
    n\log n=\sum_{r|n}\varphi'(r)
    \]
    via M\"{o}bius inversion. It follows, upon setting $n=\gcd(n_1,n_2)$, that
    \begin{align*}
        \sum_{n_i<\xi^2}\frac{\lambda_{n_1}\lambda_{n_2}}{n_1n_2}&\gcd(n_1,n_2)\log(\gcd(n_1,n_2))\\
        =&\sum_{r<\xi^2}\varphi'(r)\bigg(\sum_{\substack{n<\xi^2\\ r|n}} \frac{\lambda_n}{n}\bigg)^2\\
        =&\frac{1}{(\log\xi)^2}\Bigg\{\sum_{\xi\leq r<\xi^2}\frac{\mu(r)^2}{\varphi(r)}\bigg(\log r+\sum_{p|r}\frac{\log p}{p-1}\bigg)+O(1)\Bigg\}\\
        =&\frac{3}{2}+O\left(\frac{\log\log \xi}{\log \xi}\right),
    \end{align*}
    as desired.
\end{proof}

\subsection{Mollifying $L(s,\chi)$}
Now define the Dirichlet polynomial
\begin{equation}\label{def: psi}
    \psi(s,\chi):=\sum_{n<\xi^2}\frac{\lambda_n\chi(n)}{n^s},
\end{equation}
which may be viewed as a mollifier approximating $1/L(s,\chi)$. The next lemma roughly says that $L\cdot \psi$ is close to 1 on average to the right of the half-line.

\begin{lemma}[c.f. {\cite[Lemma 13]{Sel46}}]\label{lemma L*psi}
    Fix $0<\eps\leq 1/4$, $0<\kappa<\eps/2$, and let $\xi=q^\kappa$. Then for all $q\geq q_0(\kappa)$, $|t|\leq 2q^{1/4-\eps}$ and $\sigma\geq \frac{1}{2}+\frac{1}{8\kappa\log q}$,
    \[
    \frac{1}{q-2}\sum_{\chi\neq \chi_0}\left(|L(\sigma+it,\chi)\psi(\sigma+it,\chi)|^2-1\right)< 6.20\xi^{1-2\sigma}
    \]
    where the summation is taken over all non-principal characters modulo $q$.
\end{lemma}

\begin{proof}
    Substituting $\psi(s,\chi)$ by the defining Dirichlet polynomial \eqref{def: psi}, we have
    \begin{equation}\label{expression L*psi^2}
        \sum_{\chi\neq \chi_0} |L(\sigma+it,\chi)\psi(\sigma+it,\chi)|^2 = \sum_{n_i<\xi^2}\frac{\lambda_{n_1}\lambda_{n_2}}{n_1^{\sigma+it}n_2^{\sigma-it}}\sum_{\chi\neq \chi_0}|L(\sigma+it,\chi)|^2 \chi(n_1)\ov{\chi}(n_2).
    \end{equation}
    First suppose that $\frac{1}{2}+\frac{1}{8\kappa\log q}\leq \sigma\leq 1-\frac{1}{\log q}$. It follows from Selberg's argument in \cite[Lemma 13]{Sel46} and \cite[Theorem 1]{Sel46} that
    \begin{align*}
        \sum_{\chi\neq \chi_0} |L(\sigma+it,\chi)\psi(\sigma+it,\chi)|^2\leq & (q-1)\zeta(2\sigma)\sum_{n_i<\xi^2}\frac{\lambda_{n_1}\lambda_{n_2}}{(n_1n_2)^{2\sigma}}\gcd (n_1,n_2)^{2\sigma}\\
        &+O\bigg(\frac{1+t^2}{(1-\sigma)^2}\sum_{n_i<\xi^2}\left\{n_1^{1-\sigma}n_2^{-\sigma}q^{1-\sigma}+(n_1n_2)^{1-\sigma}q^{1-2\sigma}\right\}\bigg).
    \end{align*}
    The main term on the right-hand side is
    \begin{align*}
        (q-1)\sum_{k=1}^\infty \sum_{n_i<\xi^2}\frac{\lambda_{n_1}\lambda_{n_2}}{(k\cdot \lcm({n_1,n_2}))^{2\sigma}}=&(q-1)\sum_{m=1}^\infty m^{-2\sigma} \sum_{\substack{n_i<\xi^2\\ \lcm(n_1,n_2)|m}}\lambda_{n_1}\lambda_{n_2}\\
        =&(q-1)\sum_{m=1}^\infty m^{-2\sigma}\bigg(\sum_{\substack{n<\xi^2\\n|m}}\lambda_n\bigg)^2,
    \end{align*}
    while the error term is $O(q^{1-4(\frac{\eps}{2}-\kappa)}\xi^{1-2\sigma})=o(q\xi^{1-2\sigma})$ as $q\to \infty$. Since $\sum_{n|m}\mu(n)=1$ if $m=1$ and 0 if $m>1$,
    \begin{equation}\label{bound on L*psi 1}
        \sum_{\chi\neq \chi_0} |L(\sigma+it,\chi)\psi(\sigma+it,\chi)|^2\leq q-1+(q-1)\sum_{m>\xi} m^{-2\sigma}\bigg(\sum_{\substack{n<\xi^2\\n|m}}\lambda_n\bigg)^2+o(q\xi^{1-2\sigma}).
    \end{equation}
    Let $0<\alpha\leq (\sigma-1/2)\log q$. Note that for $m>\xi$,
    \begin{align*}
        m^{-2\sigma}=\frac{\xi^{1-2\sigma}}{\xi}\left(\frac{\xi}{m}\right)^{2\sigma}\leq \frac{\xi^{1-2\sigma}}{\xi}\left(\frac{\xi}{m}\right)^{1+\frac{2\alpha}{\log q}}= \xi^{1-2\sigma}\xi^{\frac{2\alpha}{\log q}} m^{-1-\frac{2\alpha}{\log q}}=  \xi^{1-2\sigma}e^{2\kappa\alpha}m^{-1-\frac{2\alpha}{\log q}}.
    \end{align*}
    Moreover, 
    \[
    m^{-\frac{2\alpha}{\log q}}\leq e^{-\lceil \frac{2\alpha\log m}{\log q}\rceil +1} = (e-1) \sum_{\ell\geq \lceil \frac{2\alpha\log m}{\log q}\rceil} e^{-\ell}.
    \]
    Thus 
    \begin{align}\label{bound on L*psi 2}
        \sum_{m>\xi} m^{-2\sigma}\bigg(\sum_{\substack{n<\xi^2\\n|m}}\lambda_n\bigg)^2\leq & \xi^{1-2\sigma}e^{2\kappa\alpha} \sum_{m>\xi} m^{-1-\frac{2\alpha}{\log q}}\bigg(\sum_{\substack{n<\xi^2\\n|m}}\lambda_n\bigg)^2\notag\\
        \leq & \xi^{1-2\sigma}e^{2\kappa\alpha}(e-1) \sum_{\ell=\lceil 2\alpha\kappa \rceil}^\infty e^{-\ell}\sum_{\xi<m\leq q^{\frac{\ell}{2\alpha}}} \frac{1}{m}\bigg(\sum_{\substack{n<\xi^2\\n|m}}\lambda_n\bigg)^2\notag\\
        = & \xi^{1-2\sigma}e^{2\kappa\alpha}(e-1) \sum_{\ell=\lceil 2\alpha\kappa \rceil}^\infty e^{-\ell}\Bigg\{-1+\sum_{m\leq q^{\frac{\ell}{2\alpha}}} \frac{1}{m}\bigg(\sum_{\substack{n<\xi^2\\n|m}}\lambda_n\bigg)^2\Bigg\}.
    \end{align}
    Here
    \begin{align*}
        \sum_{m\leq q^{\frac{\ell}{2\alpha}}} \frac{1}{m}\bigg(\sum_{\substack{n<\xi^2\\n|m}}\lambda_n\bigg)^2=&\sum_{m\leq q^{\frac{\ell}{2\alpha}}} \frac{1}{m}\sum_{\substack{n_i<\xi^2\\ \lcm(n_1,n_2)|m}}\lambda_{n_1}\lambda_{n_2}\\
        =&\sum_{n_i<\xi^2}\frac{\lambda_{n_1}\lambda_{n_2}}{n_1n_2}\gcd(n_1,n_2)\sum_{n\leq \frac{q^{\frac{\ell}{2\alpha}}}{\lcm(n_1,n_2)}}\frac{1}{n}.
    \end{align*}
    If $\alpha=\frac{1}{8\kappa}$ (for this we need $\sigma\geq \frac{1}{2}+\frac{1}{8\kappa \log q}$), then $q^{\frac{\ell}{2\alpha}}= q^{4\kappa\ell}=\xi^{4\ell}\geq \xi^4>\lcm(n_1,n_2)$ for each $\ell\geq \lceil 2\alpha\kappa \rceil=1$. Therefore, the above is
    \begin{align*}
        =& \sum_{n_i<\xi^2}\frac{\lambda_{n_1}\lambda_{n_2}}{n_1n_2}\gcd(n_1,n_2)\left(\log \frac{q^{4\kappa \ell}\gcd(n_1,n_2)}{n_1n_2}+O(1)\right)\\
        =& \frac{4\kappa \ell \log q}{\log \xi}+\frac{3}{2}-2+O\left(\frac{\kappa\ell\log q}{(\log\xi)^2}+\frac{\log\log\xi}{\log\xi}\right)\\
        =& 4\ell-\frac{1}{2}+O\left(\frac{\ell}{\kappa\log q}+\frac{\log\log\xi}{\log\xi}\right)
    \end{align*}
    where we inserted the three estimates from Lemma~\ref{lemma lambda_n gcd}. In view of \eqref{bound on L*psi 1} and \eqref{bound on L*psi 2}, we have
    \[
    \frac{1}{q-2}\sum_{\chi\neq \chi_0}\left(|L(\sigma+it,\chi)\psi(\sigma+it,\chi)|^2-1\right)\leq (C+o(1))\xi^{1-2\sigma}
    \]
    where 
    \begin{align*}
        C=&e^{1/4}(e-1)\sum_{\ell=1}^\infty e^{-\ell}\left(-1+4\ell-\frac{1}{2}\right)=e^{1/4}(e-1)\left(-\frac{3}{2(e-1)}+\frac{4e}{(e-1)^2}\right)\\
        =&6.199\ldots.
    \end{align*}
    This concludes the proof in the first case.
    
    For the other case, suppose that $\sigma>1-1/\log q$. Applying \cite[Lemma 7]{Sel46} to \eqref{expression L*psi^2} gives
    \begin{align*}
        \sum_{\chi\neq \chi_0} |L(\sigma+it,\chi)\psi(\sigma+it,\chi)|^2= &(q-1)\zeta(2\sigma)\sum_{n_i<\xi^2}\frac{\lambda_{n_1}\lambda_{n_2}}{(n_1n_2)^{2\sigma}}\gcd (n_1,n_2)^{2\sigma}\\
        &\quad - \bigg|\sum_{n=1}^q \frac{1}{n^{\sigma+it}}\sum_{n<\xi^2}\frac{\lambda_n}{n^{\sigma+it}}\bigg|^2\\
        &\quad + O\bigg((1+t^2)q^{1-\sigma}(\log q)^2\sum_{n_i<\xi^2}(n_1^{-\sigma}+n_2^{-\sigma})\bigg).
    \end{align*}
    The main term can be handled in the same way as before, the second term can be dropped because we are seeking an upper bound, and the error term is 
    \[
    \ll q^{3/2-\sigma-2\eps}(\log q)^2 \xi^{2+2(1-\sigma)}\log \xi
    \ll q^{3/2-\sigma-\eps/2}\xi^{1-2\sigma}=o(q\xi^{1-2\sigma}).
    \]
    The proof is now complete.
\end{proof}

The final lemma needed in this section is the following Littlewood-type result: 
\begin{lemma}[{\cite[Lemma 14]{Sel46}}]
\label{lemma littlewood}
    Let $f(s)$ be holomorphic in a neighborhood of the region $\sigma\geq \sigma'$, $t_1\leq t\leq t_2$ such that
    \begin{equation}\label{growth condition}
        f(s)=1+o(e^{-\pi\frac{\sigma-\sigma'}{t_2-t_1}})
    \end{equation}
    \footnote{In \eqref{growth condition} we have corrected Selberg's original statement where one has $\sigma$ instead of $\sigma-\sigma'$ in the exponent of the error term.} uniformly as $\sigma\to \infty$. Denote by $\beta+i\gamma$ the zeros of $f(s)$. Then
    \begin{align*}
        2(t_2-t_1)\sum_{\substack{\beta\geq \sigma'\\t_1\leq\gamma\leq t_2}} &\sin\left(\pi\frac{\gamma-t_1}{t_2-t_1}\right) \sinh\left(\pi\frac{\beta-\sigma'}{t_2-t_1}\right)\\
        =&\int_{t_1}^{t_2} \sin\left(\pi\frac{t-t_1}{t_2-t_1}\right) \log|f(\sigma'+it)|\md t\\
        &\quad+\int_{\sigma'}^\infty \sinh\left(\pi\frac{\sigma-\sigma'}{t_2-t_1}\right)(\log|f(\sigma+it_1)|+\log|f(\sigma+it_2)|)\md \sigma.
    \end{align*}
\end{lemma}

\begin{remark}\label{remark: L*psi growth condition}
    If $\xi=q^\kappa$ for some fixed $0<\kappa<\eps/2$ and $(t_2-t_1)\log q > \pi/\kappa+\delta$ for some fixed $\delta>0$, then $L(s,\chi)\psi(s,\chi)$ satisfies the growth condition \eqref{growth condition} in the preceding lemma (we may take $\sigma'=1/2$). To see this, we use the triangle inequality to trivially estimate, as $\sigma\to \infty$,
    \begin{align*}
        L(s,\chi)\psi(s,\chi)=&\sum_{m=1}^\infty \frac{\chi(m)}{m^s} \sum_{n<\xi^2} \frac{\lambda_n \chi(n)}{n^s}\\
        =& 1-\Bigg(\sum_{m=1}^\infty \frac{\chi(m)}{m^s}\Bigg)\Bigg(\sum_{\xi\leq n<\xi^2}\left(\frac{\log n}{\log \xi}-1\right)\frac{\mu(n)\chi(n)}{n^s}+\sum_{n\geq \xi^2}\frac{\mu(n)\chi(n)}{n^s}\Bigg)\\
        =& 1-o(\xi^{1-\sigma}).
    \end{align*}
    The error term is $o(e^{-\pi\frac{\sigma-1/2}{t_2-t_1}})$ as $\sigma\to \infty$ if $(t_2-t_1)\log q > \pi/\kappa+\delta$.
\end{remark}

We now proceed to the proof of our explicit zero-density result.
\subsection{Proof of Theorem~\ref{theorem zero density}}

Put 
\[
\tau=(t_2-t_1)\log q, \quad t_1'=t_1-\frac{a}{\log q}, \quad t_2'=t_2+\frac{a}{\log q}, \quad \sigma'=\sigma-\frac{b}{\log q}
\]
where $a,b$ are positive parameters to be determined. Later we will choose $a\leq \tau/2$ so that $|t_1'|$ and $|t_2'|$ are both $\leq 2q^{1/4-\eps}$, which would allow us to apply Lemma~\ref{lemma L*psi}. Further suppose that
\begin{equation}\label{condition on tau and a}
    \tau+2a>\frac{\pi}{\kappa}+\delta
\end{equation}
for some $\delta>0$. Then
\begin{align*}
    N(\sigma; t_1,t_2; \chi)\leq \sum_{\substack{\beta\geq \sigma\\t_1\leq\gamma\leq t_2}} \frac{\beta-\sigma'}{\sigma-\sigma'} \leq \frac{1}{b}\log q \sum_{\substack{\beta\geq \sigma'\\t_1\leq\gamma\leq t_2}} (\beta-\sigma').
\end{align*}
Since $\pi x<\sinh(\pi x)$ for all $x>0$ and 
\[
\sin\left(\pi\frac{\gamma-t_1'}{t_2'-t_1'}\right)\geq \sin\left(\pi \frac{t_1-t_1'}{t_2'-t_1'}\right)=\sin\left(\frac{\pi}{2} \frac{2a}{\tau+2a}\right)
\]
for all $t_1\leq \gamma \leq t_2$, we have
\[
\sin\left(\frac{\pi}{2} \frac{2a}{\tau+2a}\right)\sum_{\substack{\beta\geq \sigma'\\t_1 \leq \gamma\leq t_2}} \pi(\beta-\sigma')\leq (t_2'-t_1') \sum_{\substack{\beta\geq \sigma'\\t_1'\leq\gamma\leq t_2'}}  \sin\left(\pi\frac{\gamma-t_1'}{t_2'-t_1'}\right)\sinh\left(\pi \frac{\beta-\sigma'}{t_2'-t_1'}\right).
\]
The zeros of $L(s,\chi)$ are also zeros of $L(s,\chi)\psi(s,\chi)$, where $\psi$ is the function defined by \eqref{def: psi} with $\xi=q^{\kappa}$, $\kappa<\eps/2$. Moreover, Remark~\ref{remark: L*psi growth condition} allows us to apply Lemma~\ref{lemma littlewood} to $L(s,\chi)\psi(s,\chi)$ since $(t_2'-t_1')\log q=\tau+2a>\pi/\kappa+\delta$ by assumption \eqref{condition on tau and a}. Hence, Lemma~\ref{lemma littlewood} yields
\begin{equation}\label{N(t) bound}
    N(\sigma; t_1,t_2; \chi) \leq \frac{\log q}{2\pi b \cdot \sin(\frac{\pi}{2} \frac{2a}{\tau+2a})}\cdot \mc{I}_\chi
\end{equation}
where
\begin{align*}
    \mc{I}_\chi=&\int_{t_1'}^{t_2'} \sin\left(\pi\frac{t-t_1'}{t_2'-t_1'}\right) \log|L(\sigma'+it,\chi)\psi(\sigma'+it,\chi)|\md t\\
    &\hspace{2cm} +\int_{\sigma'}^\infty \sinh\left(\pi\frac{\alpha-\sigma'}{t_2-t_1}\right) \big(\log|L(\alpha+it_1',\chi)\psi(\alpha+it_1',\chi)|\\
    &\hspace{7cm}+\log|L(\alpha+it_2',\chi)\psi(\alpha+it_2',\chi)|\big) \md \alpha\\
    \leq &\frac{1}{2}\int_{t_1'}^{t_2'} \sin\left(\pi\frac{t-t_1'}{t_2'-t_1'}\right) \big(|L(\sigma'+it,\chi)\psi(\sigma'+it,\chi)|^2-1\big)\md t\\
    &\hspace{2cm} +\frac{1}{2}\int_{\sigma'}^\infty \sinh\left(\pi\frac{\alpha-\sigma'}{t_2'-t_1'}\right)\big(|L(\alpha+it_1',\chi)\psi(\alpha+it_1',\chi)|^2-1\\
    &\hspace{7cm}+|L(\alpha+it_2',\chi)\psi(\alpha+it_2',\chi)|^2-1\big) \md \alpha.
\end{align*}
Here we have used the inequality $\log |x|\leq \frac{|x|^2-1}{2}$, valid for all $x$. Summing over $\chi$ and applying Lemma~\ref{lemma L*psi}, we see that if $\sigma'\geq \frac{1}{2}+\frac{1}{8\kappa\log q}$, then
\begin{align*}
    \frac{1}{q-2}\sum_{\chi\neq \chi_0}\mc{I}_\chi
    <& \frac{1}{2}\int_{t_1'}^{t_2'} \sin\left(\pi\frac{t-t_1'}{t_2'-t_1'}\right)\md t \cdot 6.20\xi^{1-2\sigma'}+\frac{1}{2}\int_{\sigma'}^\infty e^{\pi \frac{\alpha-\sigma'}{t_2'-t_1'}} 6.20\xi^{1-2\alpha} \md \alpha\\
    \leq& \frac{6.20}{\pi} (t_2'-t_1')\xi^{1-2\sigma'}+ \frac{6.20}{2(2-\frac{\pi}{\kappa(\tau+2a)})\log \xi} \xi^{1-2\sigma'}\\
    =& \left(\frac{\tau+2a}{\pi} +\frac{1}{2(2-\frac{\pi}{\kappa(\tau+2a)})\kappa} \right) \frac{6.20 e^{2b\kappa} q^{\kappa(1-2\sigma)}}{\log q}.
\end{align*}
Inserting this into \eqref{N(t) bound}, we obtain
\begin{align*}
    \sum_{\chi\neq \chi_0} N(\sigma; t_1,t_2; \chi)<C'(\kappa,a,b,\tau) q^{1+\kappa(1-2\sigma)}(t_2-t_1)\log q
\end{align*}
where 
\begin{align*}
    C'(\kappa,a,b,\tau):=\frac{6.20 e^{2b\kappa }}{2\pi b \cdot \sin(\frac{\pi}{2} \frac{2a}{\tau+2a})}\left(\frac{\tau+2a}{\pi \tau}+\frac{1}{2(2-\frac{\pi}{\kappa (\tau+2a)})\kappa\tau}\right).
\end{align*}
If we ignore the second term in the bracket, which becomes negligible as $\tau\to \infty$, the entire expression is minimized when $2b\kappa=1$ and $2a/\tau= 0.82579\ldots$ (note that $a<\tau/2$, which we assumed early in the proof). Substituting the values of $a,b$ into the above expression, we find that
\[
C'(\kappa,a,b,\tau) < 4.79\kappa+\frac{4.12}{2\tau-1.73/\kappa}
\]
whenever 
\[
\sigma\geq \frac{1}{2}+\frac{1}{8\kappa\log q}+\frac{b}{\log q}=\frac{1}{2}+\frac{5}{8\kappa\log q}
\]
and
\[
1.8258\tau>\tau+2a>\frac{\pi}{\kappa}+\delta, \quad \text{or} \quad  \tau\geq \frac{1.73}{\kappa},
\]
which are the assumptions stated in the theorem. This completes the proof.

\section{Approximating $S(t,\chi)$ by Dirichlet polynomials pointwise}

Following \cite[\S 5]{Sel46}, in this section we relate $S(t,\chi)$ to a certain Dirichlet polynomial and give in Proposition~\ref{prop: approx S(t,chi)} a pointwise error estimate for the approximation. Analogous explicit calculations are done for the Riemann zeta-function in \cite{KK05}.

For $x\geq 2$, $\eta\geq 1$, define 
\begin{equation}\label{def sigma_chi}
    \sigma_{t,\chi}=\sigma_{x,t,\eta,\chi}:=\frac{1}{2}+2\cdot \max_{\substack{\beta\geq 1/2 \\ |t-\gamma|\leq \frac{x^{3(\beta-1/2)}}{\log x}}}\left\{\beta-\frac{1}{2}, \frac{\eta}{\log x}\right\},
\end{equation}
\[
\tilde{s}_{t,\chi}:=\sigma_{t,\chi}+it.
\]

\begin{lemma}\label{lemma sum 1/((sigma-beta)^2+(t-gamma)^2)}
    \[
    \sum_{\rho} \frac{\sigma_{t,\chi}-1/2}{(\sigma_{t,\chi}-\beta)^2+(t-\gamma)^2}\leq  \frac{5}{3}\left(\Re \frac{L'}{L}(\tilde{s}_{t,\chi},\chi)+ \frac{1}{2}\log q(|t|+1)+O(1)\right).
    \]
\end{lemma}
\begin{proof}
    Recall the representation formula 
    \[
    \Re\frac{L'}{L}(s,\chi)=- \frac{1}{2}\log \frac{q}{\pi}-\frac{1}{2}\Re \frac{\Gamma'}{\Gamma}\left(\frac{s+\mf{a}}{2}\right)+\sum_\rho \Re \frac{1}{s-\rho},
    \]
    where $\mf{a}=\frac{1}{2}(1-\chi(-1))$. This can be derived from the functional equation (see, e.g., the last page of \cite[\S 12]{Dav}). For $1/2\leq \sigma\leq 10$, say, the second term on the right-hand side is $-\frac{1}{2}\log(|t|+1)+O(1)$. Thus
    \[
    \sum_{\rho} \frac{\sigma_{t,\chi}-\beta}{(\sigma_{t,\chi}-\beta)^2+(t-\gamma)^2}=\Re\frac{L'}{L}(\tilde{s}_{t,\chi},\chi)+\frac{1}{2}\log q(|t|+1)+O(1).
    \]
    The rest of the proof is identical to that of \cite[Lemma 6]{KK05} which relies on an analogous formula for $\Re (\zeta'/\zeta)$ and uses only the symmetry of zeta zeros across the $1/2$-line (and not the real axis, so that the proof extends readily to $L(s,\chi)$).
\end{proof}

Let
\begin{equation}\label{def Lambda_x}
    \Lambda_x(n):=
    \begin{cases}
        \Lambda(n), & \text{if $1\leq n\leq x$},\\
        \Lambda(n)\frac{\log^2(x^3/n)-2\log^2(x^2/n)}{2\log^2 x}, & \text{if $x\leq n\leq x^2$},\\
        \Lambda(n)\frac{\log^2(x^3/n)}{2\log^2 x}, & \text{if $x^2\leq n\leq x^3$}.
    \end{cases}
\end{equation}
be a smoothed version of the von Mangoldt function. Selberg \cite[Lemma 15]{Sel46} proved that
\begin{multline}\label{L'/L}
    \frac{L'}{L}(s,\chi)=-\sum_{n<x^3}\frac{\Lambda_x(n)\chi(n)}{n^s}+\frac{1}{\log^2 x}\sum_{m=0}^\infty \frac{x^{-2m-\mf{a}-s}(1-x^{-2m-\mf{a}-s})^2}{(2m+\mf{a}+s)^2}\\
    +\frac{1}{\log^2 x}\sum_{\rho}\frac{x^{\rho-s}(1-x^{\rho-s})^2}{(s-\rho)^3},
\end{multline}
valid for $s\not \in\{-2m-\mf{a}: m\in \bb{N}_{\geq 0}\}$, $s\neq \rho$. This formula will be the basis for our approximation of $S(t,\chi)$. Let 
\[
r(x,t):=\sum_{n<x^3}\frac{\Lambda_x(n)\chi(n)}{n^{\tilde{s}_{t,\chi}}}.
\]

\begin{lemma}\label{lemma L'/L-Dirich poly}
    For $x\geq 2$ and $\sigma\geq \sigma_{t,\chi}$,
    \[
    \left|\frac{L'}{L}(s,\chi)+\sum_{n<x^3}\frac{\Lambda_x(n)\chi(n)}{n^s}\right|\leq A(\eta)x^{1/4-\sigma/2}\left(|r(x,t)|+\frac{1}{2}\log q(|t|+1)+O(1)\right)
    \]
    where
    \[
    A(\eta):=\frac{5(1+e^{-\eta})^2}{6\eta^2-5(1+e^{-\eta})^2e^{-\eta}}.
    \]
\end{lemma}
\begin{proof}
    It follows from \eqref{L'/L} that
    \begin{equation}\label{eqn L'/L-Dirich poly}
        \left|\frac{L'}{L}(s,\chi)+\sum_{n<x^3}\frac{\Lambda_x(n)\chi(n)}{n^s}\right|\leq \frac{1}{\log^2 x}\sum_{\rho}R_\rho+O\left(\frac{x^{-\sigma}}{\log^2 x}\right).
    \end{equation}
    where
    \[
    R_\rho:=\frac{x^{\beta-\sigma}(1+x^{\beta-\sigma})^2}{[(\sigma-\beta)^2+(t-\gamma)^2]^{3/2}}.
    \]
    We consider two cases. On the one hand, if $\beta\leq \frac{1}{2}(\sigma_{t,\chi}+1/2)$, then 
    \[
    \beta-\sigma\leq \frac{\sigma_{t,\chi}}{2}+\frac{1}{4}-\sigma\leq \frac{1}{4}-\frac{\sigma}{2}\leq \frac{1}{4}-\frac{\sigma_{t,\chi}}{2} \leq -\frac{\eta}{\log x},
    \]
    so that
    \[
    R_\rho \leq \frac{x^{1/4-\sigma/2}(1+x^{-\eta/\log x})^2}{(\sigma-\beta)[(\sigma_{t,\chi}-\beta)^2+(t-\gamma)^2]}\leq \frac{(1+e^{-\eta})^2}{\eta} \frac{x^{1/4-\sigma/2}\log x}{(\sigma_{t,\chi}-\beta)^2+(t-\gamma)^2}.
    \]
    On the other hand, suppose that $\beta>\frac{1}{2}(\sigma_{t,\chi}+1/2)$. Since $\beta-1/2>\frac{1}{2}(\sigma_{t,\chi}-1/2)$, this zero must not be captured by the window in the definition~\ref{def sigma_chi} of $\sigma_{t,\chi}$. (The preceding inequality also implies that $\beta-1/2>\frac{\eta}{\log x}$.) Moreover, we claim that $\beta-1/2>|\sigma_{t,\chi}-\beta|$. Indeed, If $\sigma_{t,\chi}<\beta$, then $\beta-1/2>\beta-\sigma_{t,\chi}=|\sigma_{t,\chi}-\beta|$, and if $\sigma_{t,\chi}\geq \beta$, then $\beta-1/2>\sigma_{t,\chi}-\beta=|\sigma_{t,\chi}-\beta|$. Let $y=3(\beta-1/2)\log x>3\eta\geq 3$. We therefore have
    \[
    |t-\gamma|>\frac{x^{3(\beta-1/2)}}{\log x}= \frac{e^{y}}{y} 3\left(\beta-\frac{1}{2}\right)> e^3|\sigma_{t,\chi}-\beta|,
    \]
   and thus
    \begin{align*}
        ((\sigma-\beta)^2+(t-\gamma)^2)^{3/2}\geq& |t-\gamma|(t-\gamma)^2\\
        >& \frac{x^{3(\beta-1/2)}}{\log x}\frac{e^6}{e^6+1}[(\sigma_{t,\chi}-\beta)^2+(t-\gamma)^2].
    \end{align*}
    Further note that
    \begin{align*}
        \frac{x^{\beta-\sigma}(1+x^{\beta-\sigma})^2}{x^{3(\beta-1/2)}}
        =& x^{1/4-\sigma/2}x^{1/4-\sigma/2}x^{-2(\beta-1/2)}(1+x^{\beta-\sigma})^2\\
        \leq & x^{1/4-\sigma/2} e^{-\eta} x^{-2(\beta-1/2)}(1+x^{\beta-1/2})^2\\
        \leq& x^{1/4-\sigma/2} e^{-\eta} (1+x^{1/2-\beta})^2\\
        \leq& x^{1/4-\sigma/2}(1+e^{-\eta})^2 e^{-\eta}.
    \end{align*}
    Hence, 
    \begin{align*}
        R_\rho\leq& \frac{(1+e^{-\eta})^2}{e^{\eta}} \frac{e^6+1}{e^6} \frac{x^{1/4-\sigma/2}\log x}{(\sigma_{t,\chi}-\beta)^2+(t-\gamma)^2}\\
        <& \frac{(1+e^{-\eta})^2}{\eta} \frac{x^{1/4-\sigma/2}\log x}{(\sigma_{t,\chi}-\beta)^2+(t-\gamma)^2},
    \end{align*}
    so that the last estimate holds in either case (i.e., for all $\rho$). Inserting this back into \eqref{eqn L'/L-Dirich poly} and invoking Lemma~\ref{lemma sum 1/((sigma-beta)^2+(t-gamma)^2)}, we see that 
    \begin{align*}
        &\left|\frac{L'}{L}(s,\chi)+\sum_{n<x^3}\frac{\Lambda_x(n)\chi(n)}{n^s}\right|\\ 
        \leq& \frac{(1+e^{-\eta})^2}{\eta} \frac{x^{1/4-\sigma/2}}{\log x} \sum_\rho \frac{1}{(\sigma_{t,\chi}-\beta)^2+(t-\gamma)^2} + O\left(\frac{x^{-\sigma}}{\log^2 x}\right)\\
        \leq& \frac{5(1+e^{-\eta})^2}{3\eta} \frac{x^{1/4-\sigma/2}}{\log x} \left(\sigma_{t,\chi}-\frac{1}{2}\right)^{-1} \left(\Re \frac{L'}{L}(\tilde{s}_{t,\chi},\chi)+ \frac{1}{2}\log q(|t|+1)+O(1)\right)\\
        \leq & \frac{5(1+e^{-\eta})^2}{6\eta^2} x^{1/4-\sigma/2} \left(\left| \frac{L'}{L}(\tilde{s}_{t,\chi},\chi)\right|+ \frac{1}{2}\log q(|t|+1)+O(1)\right).
    \end{align*}
    Applying this estimate first to $\sigma=\sigma_{t,\chi}$ yields
    \begin{multline*}
    \left(1- \frac{5(1+e^{-\eta})^2}{6\eta^2} x^{1/4-\sigma_{t,\chi}/2}\right) \left|\frac{L'}{L}(\tilde{s}_{t,\chi},\chi)\right|\\
    \leq |r(x,t)|+ \frac{5(1+e^{-\eta})^2}{12\eta^2} x^{1/4-\sigma_{t,\chi}/2}\log q(|t|+1)+O(1).
    \end{multline*}
    The factor on the left-hand side is at least
    \[
    1-\frac{5(1+e^{-\eta})^2 }{6\eta^2 e^{\eta}}>0,
    \]
    so it follows that
    \begin{equation}\label{L'/L at sigma_t,chi}
        \left|\frac{L'}{L}(\tilde{s}_{t,\chi},\chi)\right|\leq \frac{1}{1-\frac{5(1+e^{-\eta})^2 }{6\eta^2 e^{\eta}}}|r(x,t)|
        +\frac{5(1+e^{-\eta})^2}{12\eta^2 e^{\eta}-10(1+e^{-\eta})^2}\log q(|t|+1)+O(1).
    \end{equation}
    We obtain the stated bound in the lemma upon combining this with the previous estimate. 
\end{proof}

\begin{corollary}\label{corollary sum 1/((sigma-beta)^2+(t-gamma)^2}
    \[
    \sum_{\rho} \frac{\sigma_{t,\chi}-1/2}{(\sigma_{t,\chi}-\beta)^2+(t-\gamma)^2}\leq  \frac{5}{3}\frac{1}{1-\frac{5(1+e^{-\eta})^2}{6\eta^2 e^\eta}} \left(|r(x,t)|+\frac{1}{2}\log q(|t|+1)+O(1)\right).
    \]
\end{corollary}
\begin{proof}
    This follows by inserting \eqref{L'/L at sigma_t,chi} into Lemma~\ref{lemma sum 1/((sigma-beta)^2+(t-gamma)^2)}.
\end{proof}

The main result in this section is the following:
\begin{proposition}[c.f. {\cite[Theorem 5]{Sel46}}]\label{prop: approx S(t,chi)}
    For $x\geq 2$,
    \begin{multline*}
        \left|S(t,\chi)-\frac{1}{\pi}\Im \sum_{n<x^3}\frac{\Lambda_x(n)\chi(n)}{n^{\tilde{s}_{t,\chi}}\log n}\right|\\
        \leq \frac{1}{\pi}\left(\sigma_{t,\chi}-\frac{1}{2}\right) \left(B_1(\eta)|r(x,t)|
        +B_2(\eta)\log q(|t|+1)+O(1)\right)
    \end{multline*}
    where
    \[
    B_1(\eta):=\frac{5(1+e^{-\eta})^2+(5\pi+1)6\eta^3 e^{\eta}}{6\eta^3 e^{\eta}-5(1+e^{-\eta})^2\eta}, \:\: B_2(\eta):=\frac{5(1+e^{-\eta})^2(1+\eta)+30\pi\eta^3 e^{\eta}}{12\eta^3 e^{\eta}-10(1+e^{-\eta})^2\eta}.
    \]
\end{proposition}

\begin{proof}
    We start by decomposing $S(t,\chi)$ into three terms:
    \begin{equation}
        \pi S(t,\chi)=-\int_{1/2}^\infty \Im \frac{L'}{L}(\sigma+it,\chi)\md \sigma=J_1+J_2+J_3,
    \end{equation}
    where
    \begin{align*}
        J_1:=&-\int_{\sigma_{t,\chi}}^\infty \Im \frac{L'}{L}(\sigma+it,\chi)\md \sigma, \quad J_2:=-\left(\sigma_{t,\chi}-\frac{1}{2}\right)\Im \frac{L'}{L}(\tilde{s}_{t,\chi},\chi),\\
        J_3:=&\int_{1/2}^{\sigma_{t,\chi}} \Im\left\{\frac{L'}{L}(\sigma_{t,\chi}+it,\chi)-\frac{L'}{L}(\sigma+it,\chi)\right\}\md \sigma.
    \end{align*}
    Lemma~\ref{lemma L'/L-Dirich poly} implies that
    \begin{multline*}
        \left|J_1-\int_{\sigma_{t,\chi}}^\infty  \Im \sum_{n<x^3}\frac{\Lambda_x(n)\chi(n)}{n^{\sigma+it}}\md \sigma\right|\\
        \leq A(\eta)\left(|r(x,t)|+\frac{1}{2}\log q(|t|+1)+O(1)\right) \int_{\sigma_{t,\chi}}^\infty x^{1/4-\sigma/2} \md\sigma.
    \end{multline*}
    The last integral is
    \[
    =\frac{2}{\log x}x^{1/4-\sigma_{t,\chi}/2}\leq \frac{2}{e^{\eta}\log x}\leq \frac{\sigma_{t,\chi}-1/2}{\eta e^{\eta}},
    \]
    and thus
    \[
    \left|J_1-\Im \sum_{n<x^3}\frac{\Lambda_x(n)\chi(n)}{n^{\tilde{s}_{t,\chi}}\log n}\right|\leq \frac{A(\eta)}{\eta e^{\eta}}\left(\sigma_{t,\chi}-\frac{1}{2}\right)\left(|r(x,t)|+\frac{1}{2}\log q(|t|+1)+O(1)\right).
    \]
    The term $|J_2|$ can be bounded directly by \eqref{L'/L at sigma_t,chi}. For $|J_3|$, observe from \cite[\S 12, display (17)]{Dav} that
    \begin{multline*}
        \Im\left\{\frac{L'}{L}(\tilde{s}_{t,\chi},\chi)-\frac{L'}{L}(s,\chi)\right\}=-\frac{1}{2}\Im\left\{\frac{\Gamma'}{\Gamma}\left(\frac{\tilde{s}_{t,\chi}+\mf{a}}{2}\right)-\frac{\Gamma'}{\Gamma}\left(\frac{s+\mf{a}}{2}\right)\right\}\\
        +\sum_\rho \Im \left\{\frac{1}{\tilde{s}_{t,\chi}-\rho}-\frac{1}{s-\rho}\right\}.
    \end{multline*}
    Then, according to Selberg's argument in \cite[pp. 51--52]{Sel46}, we have
    \begin{align*}
        |J_3|\leq& 3\pi \left(\sigma_{t,\chi}-\frac{1}{2}\right)^2 \sum_\rho \frac{1}{(\sigma_{t,\chi}-\beta)^2+(t-\gamma)^2}+O\left(\sigma_{t,\chi}-\frac{1}{2}\right)\\
        \leq& \left(\sigma_{t,\chi}-\frac{1}{2}\right)\frac{5\pi}{1-\frac{5(1+e^{-\eta})^2}{6\eta^2 e^{\eta}}}\left(|r(x,t)|+\frac{1}{2}\log q(|t|+1)+O(1)\right)
    \end{align*}
    where we applied Corollary~\ref{corollary sum 1/((sigma-beta)^2+(t-gamma)^2} in the second inequality. Collecting these estimates gives the proposition.
\end{proof}

\section{Approximating $S(t,\chi)$ by Dirichlet polynomials on average}

Using the preliminary results proved in the last two sections, we further develop our approximation of $S(t,\chi)$ by taking higher moments and averaging over $\chi\bmod q$, leading to Proposition~\ref{prop: selberg theorem 7}. We start with a general mean value estimate involving a prime sum twisted by $\chi$.

\begin{lemma}\label{lemma square prime sum}
Let $\{a_p\}$ be any sequence of complex numbers, $\ell,k\in \bb{N}$ and $2\leq y\leq q^{\frac{1}{\ell k}}$. Then, as $q\to \infty$,
\[
\frac{1}{q-2}\sum_{\chi\neq \chi_0}\left|\sum_{p\leq y}\frac{a_p\chi(p^\ell)}{p^\sigma}\right|^{2k}\leq (1+o(1))k!\left(\sum_{p\leq y}\frac{|a_p|^2}{p^{2\sigma}}\right)^k.
\]
\end{lemma}

\begin{proof}
    If we write
    \[
    \left(\sum_{p\leq y}\frac{a_p\chi(p)}{p^\sigma}\right)^{k}=\sum_{n\leq y^k}\frac{b_n\chi(n)}{n^\sigma},
    \]
    then from the orthogonality relation
    \begin{equation}\label{orthogonality}
        \sum_{\chi\bmod q}\chi(n)=
        \begin{cases}
            q-1 & \text{if $n\equiv 1\bmod q$},\\
            0 & \text{otherwise}
        \end{cases}
    \end{equation}
    and the restriction $y^k\leq q$, we deduce that
    \begin{align*}
        \frac{1}{q-2}\sum_{\chi\neq \chi_0}\left|\sum_{p\leq y}\frac{a_p\chi(p)}{p^\sigma}\right|^{2k}=&\frac{1}{q-2}\sum_{2\leq m,n\leq y^k} \frac{b_m \ov{b_n}}{(mn)^{\sigma}}\sum_{\chi\neq \chi_0}\chi(m)\ov{\chi(n)}\\
        =&\frac{q-1}{q-2}\sum_{2\leq n\leq y^k} \frac{|b_n|^2}{n^{2\sigma}}-\frac{1}{q-2}\left|\sum_{2\leq n\leq y^k}\frac{b_n\chi_0(n)}{n^\sigma}\right|^2 \\
        \leq & (1+o(1))\sum_{2\leq n\leq y^k} \frac{|b_n|^2}{n^{2\sigma}}.
    \end{align*}
    Since $b_n=\sum_{p_1\ldots p_k=n}a_{p_1}\ldots a_{p_k}$,
    by the Cauchy--Schwarz inequality,
    \[
    \frac{|b_n|^2}{n^{2\sigma}}\leq \sum_{p_1\ldots p_k=n} \frac{|a_{p_1}\ldots a_{p_n}|^2 }{(p_1\ldots p_k)^{2\sigma}}P(p_1,\ldots,p_k)
    \]
    where $P(p_1,\ldots,p_k)$ denotes the number of permutations of the prime tuple, which is at most $k!$. The desired estimate for $\ell=1$ follows immediately, and the case for a general $\ell$ can be treated in a similar way.
\end{proof}

In the next two lemmas we exploit Theorem~\ref{theorem zero density} to derive mean value estimates involving the quantity $\sigma_{t,\chi}$ defined in \eqref{def sigma_chi}. 

\begin{lemma}\label{lemma 17 v2}
    Let $0<\eps \le 1/4$, $0<\kappa<\eps/2$ and $|t|\le q^{\frac{1}{4}-\eps}$. Further, let $v,r\in \bb{R}_{\geq 0}$, $0<\delta< 2/(2r+3)$ and set $x=q^{\delta\kappa}$. Then for all $q\ge q_0(\kappa)$,
    $$\frac{1}{q-2}\sum_{\chi\ne \chi_0} \left(\sigma_{t,\chi}-\frac{1}{2}\right)^v x^{r\left(\sigma_{t,\chi}-\frac{1}{2}\right)}\le   \frac{C(\eta,\delta,r,v)}{(\log x)^v},$$
    where $\eta$ is the parameter appearing in the definition of $\sigma_{t,\chi}$ in \eqref{def sigma_chi}, and
    \begin{align*}
        C(\eta,\delta,r,v):=(2\eta)^ve^{2\eta r}+&2^{v+1}\left(4.79+\frac{4.12}{4e^{3\eta}/\delta-1.73}\right)\\
        \cdot& \left[\frac{\eta^v e^{(2r+3-2/\delta)\eta}}{\delta}+\frac{v\Gamma(v,\eta(2/\delta-2r-3))}{\delta(2/\delta-2r-3)^v}+\frac{2r\Gamma(v+1,\eta(2/\delta-2r-3))}{\delta(2/\delta-2r-3)^{v+1}}\right].
    \end{align*}
    Here $\Gamma(v,u)=\int_u^\infty z^{v-1}e^{-z}\md z$ is the incomplete Gamma function.
\end{lemma}

\begin{proof}
Let us denote 
$$\sigma_0=\frac{1}{2}+\frac{\eta}{\log x}\qquad \text{and}\qquad \Delta(\beta) = \frac{x^{3(\beta-1/2)}}{\log x}.$$
Recalling the definition of $\sigma_{t,\chi}$ in \eqref{def sigma_chi}, we have
\begin{equation*}
    \sigma_{t,\chi} = \frac{1}{2} + 
    2 \max_{\substack{\beta\ge \frac{1}{2} \\ |t-\gamma|\leq \Delta(\beta)}}\left(\beta-\frac{1}{2},\sigma_0-\frac{1}{2}\right).
\end{equation*}
Considering the cases $\beta \le \sigma_0$ and $\beta>\sigma_0$ gives the formula
\begin{equation*}
    \sigma_{t,\chi} = \frac{1}{2}+ 2\cdot 
    \begin{cases}
        \displaystyle
        \frac{\eta}{\log x} & \text{if $\beta\le \sigma_0$ throughout;}\\
        \displaystyle
        \max_{\substack{\beta > \sigma_0 \\ |t-\gamma|\leq \Delta(\beta)}}\left(\beta-\frac{1}{2}\right) & \text{otherwise},
    \end{cases}
\end{equation*}
which implies that
\begin{equation}\label{selberg lemma 17 v2 s1}
    \left(\sigma_{t,\chi}-\frac{1}{2}\right)^v x^{r\left(\sigma_{t,\chi}-\frac{1}{2}\right)} \le  \frac{(2\eta)^v e^{2\eta r}}{(\log x)^v} +2^v\sum_{\substack{\beta >\sigma_0\\ |\gamma-t|\le \Delta(\beta)}} \left(\beta-\frac{1}{2}\right)^v x^{2r\left(\beta-\frac{1}{2}\right)}.
\end{equation}
The above sum over zeros can be written as a Stieltjes integral and equals, via integration by parts,
\[
-\int_{\sigma_0}^1 \left(\sigma-\frac{1}{2}\right)^v x^{2r(\sigma-1/2)} \md N(\sigma;t-\Delta(\sigma),t+\Delta(\sigma); \chi)=A_\chi+B_\chi
\]
where
\[
A_{\chi}:=\left(\sigma_0-\frac{1}{2}\right)^v x^{2r(\sigma_0-1/2)}N(\sigma_0;t-\Delta(\sigma_0),t+\Delta(\sigma_0); \chi)
\]
and
\[
B_{\chi}:=\int_{\sigma_0}^1 \left(v\left(\sigma-\frac{1}{2}\right)^{v-1}+ \left(\sigma-\frac{1}{2}\right)^v  2r\log x\right) x^{2r(\sigma-1/2)} N(\sigma;t-\Delta(\sigma),t+\Delta(\sigma); \chi)\md \sigma.
\]
We shall sum over $\chi\ne \chi_0$ and use the zero-density inequality furnished by Theorem~\ref{theorem zero density} to bound the resulting expressions. In preparation,
note that $x< q^{\kappa}$ and $\eta\ge 1$, and hence for all $\sigma\geq \sigma_0$,
$$\sigma\ge \frac{1}{2}+\frac{\eta}{\log x} > \frac{1}{2}+ \frac{5}{8\kappa \log q} \qquad \text{and}\qquad  \Delta(\sigma)\ge \Delta(\sigma_0)= \frac{e^{3\eta}}{\log x}>  \frac{1}{2}\cdot \frac{1.73}{\kappa \log q}.$$
This verifies that the requisite conditions to apply Theorem~\ref{theorem zero density} are fulfilled. Moreover, to simplify the right-hand side in Theorem~\ref{theorem zero density}, let us denote
\begin{equation}\label{f(kappa, delta)}
    f(\eta, \kappa,\delta) := 2\kappa \left(4.79 + \frac{4.12}{4 e^{3\eta}/\delta-1.73} \right).
\end{equation}
Then for $\sigma\geq \sigma_0$, Theorem~\ref{theorem zero density} gives
\begin{align*}
   \frac{1}{q-2}\sum_{\chi\ne \chi_0} N(\sigma; t-\Delta(\sigma),t+\Delta(\sigma); \chi) \le& f(\eta, \kappa,\delta) \log q\cdot \Delta(\sigma) q^{-2\kappa (\sigma-1/2)}\\
   \leq& f(\eta, \kappa,\delta) \frac{\log q}{\log x}\left(\frac{x^3}{q^{2\kappa}}\right)^{\sigma-1/2}\\
   =&\frac{f(\eta, \kappa,\delta)}{\delta\kappa}x^{(3-2/\delta)(\sigma-1/2)}.
\end{align*}
From this we obtain
\begin{equation}\label{lemma 17 A_chi}
    \frac{1}{q-2}\sum_{\chi\neq \chi_0}A_\chi\leq \frac{f(\eta, \kappa,\delta)}{\delta\kappa}\frac{\eta^v e^{2\eta r}}{(\log x)^v}e^{\eta(3-2/\delta)}
\end{equation}
and
\begin{equation}\label{lemma 17 B_chi}
    \begin{split}
        \frac{1}{q-2}\sum_{\chi\neq \chi_0}B_\chi\leq  \frac{vf(\eta, \kappa,\delta)}{\delta\kappa}\int_{\sigma_0}^1 &\left(\sigma-\frac{1}{2}\right)^{v-1}x^{(2r+3-2/\delta)(\sigma-1/2)}\md \sigma\\
        +& \frac{f(\eta, \kappa,\delta)}{\delta\kappa}2r\log x \int_{\sigma_0}^1 \left(\sigma-\frac{1}{2}\right)^v x^{(2r+3-2/\delta)(\sigma-1/2)}\md \sigma.
    \end{split}
\end{equation}
A change of variable $\sigma-1/2\mapsto \alpha$ transforms the first integral in \eqref{lemma 17 B_chi} into
\begin{align*}
    \int_{\sigma_0-1/2}^{1/2} \alpha^{v-1} x^{(2r+3-2/\delta)\alpha}\md \alpha<\int_{\sigma_0-1/2}^{\infty} \alpha^{v-1} x^{(2r+3-2/\delta)\alpha}\md \alpha.
\end{align*}
After applying another change of variable $\log x\cdot (2/\delta-2r-3)\alpha\mapsto z$ and recalling the definition of the incomplete Gamma function, we see that the last integral equals
\[
\frac{\Gamma(v,\eta(2/\delta-2r-3))}{(\log  x)^v (2/\delta-2r-3)^v}.
\]
Here we used our assumption that $(2r+3)\delta<2$. Similarly, the second integral in \eqref{lemma 17 B_chi} is at most
\[
\frac{\Gamma(v+1,\eta(2/\delta-2r-3))}{(\log  x)^{v+1}(2/\delta-2r-3)^{v+1}}.
\]
Combining these estimates and substituting \eqref{lemma 17 A_chi} and \eqref{lemma 17 B_chi} back into \eqref{selberg lemma 17 v2 s1} yields the result in the lemma.
\end{proof}

When $\Re(s)>1$ we have the Dirichlet series representation
\[
\log L(s,\chi)=\sum_{n=1}^\infty \frac{\Lambda(n)\chi(n)}{n^s \log n}.
\]
This leads us to the heuristic 
\[
S(t,\chi)=\frac{1}{\pi}\Im \log L\left(\frac{1}{2}+it,\chi\right)\approx \frac{1}{\pi}\Im \sum_{p<X}\frac{\chi(p)}{p^{1/2+it}}
\]
for large $X$. We make this more precise in the following result by providing a mean value estimate for even moments of the error term in this approximation.

\begin{proposition}[c.f. {\cite[Theorem 7]{Sel46}}]\label{prop: selberg theorem 7}
    Let $0<\eps \le 1/4$, $0<\kappa<\eps/2$, and $|t|\le q^{\frac{1}{4}-\eps}$. Further, let $k\in \bb{N}_{\geq 1}$, $0<\delta<2/(8k+3)$, and $x=q^{\delta\kappa}$. Then for $q\geq q_0(\delta,\kappa,k)$,
    \begin{equation*}
        \frac{1}{q-2}\sum_{\chi\neq \chi_0} \left|S(t,\chi)-\frac{1}{\pi}\Im \sum_{p<x^3}\frac{\chi(p)}{p^{1/2+it}}\right|^{2k} < D(\eta,\delta,\kappa,k,\eps),
    \end{equation*}
    where $\eta\geq 1$, 
    \begin{align*}
        D(\eta,\delta,\kappa,k,\eps):=&
        \frac{1}{\pi^{2k}} \Bigg(\left(B_1(\eta)+e^{-2\eta}\right) C(\eta,\delta,4k,4k)^{\frac{1}{4k}}((2k)!)^{\frac{1}{4k}} \sqrt{h(k)}+\frac{3}{2}(1+B_1(\eta))C(\eta,\delta,0,2k)^{\frac{1}{2k}}\\
        &\hspace{1cm}+B_2(\eta)C(\eta,\delta,0,2k)^{\frac{1}{2k}}\frac{5/4-\eps}{\delta \kappa}+(k!)^{\frac{1}{2k}}\sqrt{0.12}+0.53+(k!)^{\frac{1}{2k}}\sqrt{0.58}\Bigg)^{2k},
    \end{align*}
    \begin{equation}\label{def h(k)}
        h(k):=\min_{\Delta\geq 0}\left[(1-e^{-\Delta})8.68^{2k}+e^{-\Delta}\left(\frac{\left(2\Delta^{2}+4\Delta+3\right)e^{6\Delta}-192\Delta^{3}-80\Delta^{2}-22\Delta-3}{8\Delta^{4}e^{6\Delta}}\right)^{2k}\right]^{\frac{1}{2k}},
    \end{equation}
    $B_1$ and $B_2$ were defined in Proposition~\ref{prop: approx S(t,chi)}, and $C$ was defined in Lemmas~\ref{lemma 17 v2}.
\end{proposition}

We stated and shall prove the Proposition for any positive integer $k$, but in our applications only the case $k=1$ is needed. In particular, $h(1)=6.390\ldots$ with the minimum attained at $\Delta=0.645\ldots$. 

\begin{proof}
    We begin by writing
    \begin{align*}
        \pi S(t,\chi)-\Im \sum_{p<x^3}\frac{\chi(p)}{p^{1/2+it}}
        = \pi S(t,\chi)-&\Im \sum_{n<x^3}\frac{\Lambda_x(n)\chi(n)}{n^{\sigma_{t,\chi}+it}\log n}\\
        +&\Im \sum_{n<x^3}\frac{\Lambda_x(n)\chi(n)}{n^{\sigma_{t,\chi}+it}\log n}-\Im \sum_{n<x^3}\frac{\Lambda_x(n)\chi(n)}{n^{1/2+it}\log n}\\
        +&\Im \sum_{n<x^3}\frac{\Lambda_x(n)\chi(n)}{n^{1/2+it}\log n}-\Im \sum_{p<x^3}\frac{\Lambda_x(p)\chi(p)}{p^{1/2+it}\log p}\\
        +&\Im \sum_{p<x^3}\frac{(\Lambda_x(p)-\Lambda(p))\chi(p)}{p^{1/2+it}\log p}.
    \end{align*}
    Label the terms in each line on the right-hand side by $I_1,I_2,I_3,I_4$. By Proposition~\ref{prop: approx S(t,chi)},
    \begin{align*}
        |I_1|\leq& \left(\sigma_{t,\chi}-\frac{1}{2}\right) B_1(\eta) \left(\left|\sum_{p<x^3}\frac{\Lambda_x(p)\chi(p)}{p^{\sigma_{t,\chi}+it}}\right|+\left|\sum_{p<x^{3/2}}\frac{\Lambda_x(p^2)\chi(p^2)}{p^{2\sigma_{t,\chi}+2it}}\right|\right)\\
        &\hspace{1cm}+ \left(\sigma_{t,\chi}-\frac{1}{2}\right) B_2(\eta)\log q(|t|+1)+O\left(\sigma_{t,\chi}-\frac{1}{2}\right).
    \end{align*}
    The first sum here can be rewritten as
    \begin{align}\label{Lambda_x(p) integral}
        \left|\sum_{p<x^3}\frac{\Lambda_x(p)\chi(p)}{p^{\sigma_{t,\chi}+it}}\right|= &x^{\sigma_{t,\chi-1/2}}\left|\int_{\sigma_{t,\chi}}^\infty x^{1/2-\sigma}\sum_{p<x^3}\frac{\Lambda_x(p)\log(xp)\chi(p)}{p^{\sigma+it}}\md \sigma\right|\notag \\
        \leq& x^{\sigma_{t,\chi-1/2}}\int_{1/2}^\infty x^{1/2-\sigma}\left|\sum_{p<x^3}\frac{\Lambda_x(p)\log(xp)\chi(p)}{p^{\sigma+it}}\right|\md \sigma,
    \end{align}
    and the second sum is at most
    \begin{equation}\label{sum logp/p}
        \sum_{p<x^{3/2}}\frac{\log p}{p}=\frac{3}{2}\log x+O(1)
    \end{equation}
    by Mertens' theorem. 
    Next, since $1-e^{-x}<x$ for all $x$, 
    \begin{align*}
        |I_2|=&\left|\sum_{p<x^3} \frac{\Lambda_x(p)\chi(p)}{p^{1/2+it}\log p}\left(1-p^{1/2-\sigma_{t,\chi}}\right)+\frac{1}{2}\sum_{p<x^{3/2}} \frac{\Lambda_x(p^2)\chi(p^2)}{p^{1+2it}\log p}\left(1-p^{1-2\sigma_{t,\chi}}\right)\right|+O\left(\sigma_{t,\chi}-\frac{1}{2}\right)\\
        \leq& \left(\sigma_{t,\chi}-\frac{1}{2}\right)\left|\sum_{p<x^3} \frac{\Lambda_x(p)\chi(p)}{p^{1/2+it}}\right| + \frac{3}{2}\left(\sigma_{t,\chi}-\frac{1}{2}\right)\log x
        +O\left(\sigma_{t,\chi}-\frac{1}{2}\right)
    \end{align*}
    where we applied \eqref{sum logp/p} again. By a similar reasoning to \eqref{Lambda_x(p) integral}, the first term above is at most
    \[
    \left(\sigma_{t,\chi}-\frac{1}{2}\right)\int_{1/2}^\infty x^{1/2-\sigma}\left|\sum_{p<x^3}\frac{\Lambda_x(p)\log(xp)\chi(p)}{p^{\sigma+it}}\right|\md \sigma.
    \]
    For $I_3$, note that
    \[
    \sum_{\ell=3}^\infty \sum_{p^\ell<x^3} \frac{\Lambda_x(p^\ell)}{p^{\ell/2}\log p^\ell} < \sum_{\ell=3}^\infty \sum_p \frac{1}{\ell p^{\ell/2}}=\sum_{\ell=3}^\infty  \frac{\zeta_\mathcal{P}(\ell/2)}{\ell}
    \]
    where $\zeta_\mathcal{P}$ is the prime zeta function. We verify that this infinite sum does not exceed 0.53 by computing on Mathematica its partial sum up to $\ell=500$ and then bounding the tail of the series by
    \[
    \sum_{\ell=501}^\infty \frac{\zeta(\ell/2)-1}{\ell}< \sum_{\ell=501}^\infty \frac{1}{\ell(\ell/2-1)}=\frac{1}{499}+\frac{1}{500}.
    \]
    Here we used the inequality $\zeta(\sigma)<\frac{\sigma}{\sigma-1}$ for $\sigma>1$ (see, e.g., \cite[Corollary 1.14]{MV})). Hence 
    \[
    |I_3|\leq \frac{1}{2}\left|\sum_{p<x^{3/2}} \frac{\Lambda_x(p^2)\chi(p^2)}{p^{1+2it}\log p}\right|+0.53.
    \]
    Putting everything together, we arrive at
    \begin{equation}\label{eqn approx S(t,chi)}
        \begin{split}
            &\left|\pi S(t,\chi)-\Im \sum_{p<x^3}\frac{\chi(p)}{p^{1/2+it}}\right|\\
            &\hspace{0.5cm}\leq \left(\sigma_{t,\chi}-\frac{1}{2}\right)\left(1+B_1(\eta)x^{\sigma_{t,\chi-1/2}}\right)\int_{1/2}^\infty x^{1/2-\sigma}\left|\sum_{p<x^3}\frac{\Lambda_x(p)\log(xp)\chi(p)}{p^{\sigma+it}}\right|\md \sigma\\
            &\hspace{1cm}+\frac{3}{2}(1+B_1(\eta))\left(\sigma_{t,\chi}-\frac{1}{2}\right)\log x+B_2(\eta)\left(\sigma_{t,\chi}-\frac{1}{2}\right)\log q(|t|+1)\\
            &\hspace{1cm}+\frac{1}{2}\left|\sum_{p<x^{3/2}} \frac{\Lambda_x(p^2)\chi(p^2)}{p^{1+2it}\log p}\right|+0.53+ \sum_{p<x^3}\left|\frac{(\Lambda_x(p)-\Lambda(p))\chi(p)}{p^{1/2+it}\log p}\right|+O\left(\sigma_{t,\chi}-\frac{1}{2}\right).
        \end{split}
    \end{equation}
    It follows from Minkowski's inequality (for $\ell_{p}$ norm, $p=2k$) that
    \begin{equation}\label{Minkoski}
        \frac{1}{q-2}\sum_{\chi\neq \chi_0}\left|\pi S(t,\chi)-\Im \sum_{p<x^3}\frac{\chi(p)}{p^{1/2+it}}\right|^{2k}\leq \left(\sum_{i=1}^7|A_i|^{\frac{1}{2k}}\right)^{2k},
    \end{equation}
    where
    \begin{align*}
        A_1:=&\frac{1}{q-2}\sum_{\chi\neq \chi_0} \left(\sigma_{t,\chi}-\frac{1}{2}\right)^{2k}\left(1+B_1(\eta)x^{\sigma_{t,\chi-1/2}}\right)^{2k} \\
        &\hspace{4cm}\cdot \left(\int_{1/2}^\infty x^{1/2-\sigma}\left|\sum_{p<x^3}\frac{\Lambda_x(p)\log(xp)\chi(p)}{p^{\sigma+it}}\right|\md \sigma\right)^{2k},\\
        A_2:=&\frac{1}{q-2}\sum_{\chi\neq \chi_0}\left(\frac{3}{2}(1+B_1(\eta))\right)^{2k}\left(\sigma_{t,\chi}-\frac{1}{2}\right)^{2k}(\log x)^{2k},\\        
        A_3:=&\frac{1}{q-2}\sum_{\chi\neq \chi_0}B_2(\eta)^{2k} \left(\sigma_{t,\chi}-\frac{1}{2}\right) ^{2k}\left(\log q(|t|+1)\right)^{2k},\\        
        A_4:=&\frac{1}{q-2}\sum_{\chi\neq \chi_0} \frac{1}{2^{2k}} \left|\sum_{p<x^{3/2}} \frac{\Lambda_x(p^2)\chi(p^2)}{p^{1+2it}\log p}\right|^{2k},\\       
        A_5:=&0.53^{2k},\\
        A_6:=&\frac{1}{q-2}\sum_{\chi\neq\chi_0}\sum_{p<x^3}\left|\frac{(\Lambda_x(p)-\Lambda(p))\chi(p)}{p^{1/2+it}\log p}\right|^{2k},\\        
        A_7:=&\frac{c(k)}{q-2}\sum_{\chi\neq \chi_0}\left(\sigma_{t,\chi}-\frac{1}   {2}\right)^{2k}, \quad \text{$c(k)$ is some constant.}
    \end{align*}   
    We first bound $A_2$ to $A_7$. By taking $r=0$ and $v=2k$ in Lemma~\ref{lemma 17 v2} we have $|A_7|\ll (\log x)^{-2k}=o(1)$,
    \[
    |A_2|\leq \left(\frac{3}{2}(1+B_1(\eta))\right)^{2k}C(\eta,\delta,0,2k),
    \]
    and
    \[
    |A_3|\leq B_2(\eta)^{2k}C(\eta,\delta,0,2k)\left(\frac{\log q(|t|+1)}{\log x}\right)^{2k}\leq B_2(\eta)^{2k}C(\eta,\delta,0,2k)\left(\frac{5/4-\eps}{\delta\kappa}+o(1)\right)^{2k},
    \]
    where we used the assumptions $|t|\leq q^{1/4-\eps}$ and $x=q^{\delta \kappa}$. Since $x^3< q^{1/k}$, Lemma~\ref{lemma square prime sum} implies that
    \begin{align*}
        |A_4| \leq& \frac{(1+o(1))k!}{4^k}\left(\sum_{p<x^{3/2}} \frac{\Lambda_x(p^2)^2}{p^2(\log p)^2}\right)^k < \frac{(1+o(1))k!}{4^k}\left(\sum_{p} \frac{1}{p^2}\right)^k<k!\cdot 0.12^k .
    \end{align*}
    and
    \begin{align*}
        |A_6|\leq& (1+o(1))k!\left(\sum_{p<x^3}\frac{(\Lambda_x(p)-\Lambda(p))^2}{p(\log p)^2}\right)^{k}<(1+o(1))k!\left(\sum_{x\leq p<x^2}\frac{1}{4p}+\sum_{x^2\leq p<x^3}\frac{1}{p}\right)^{k}\\
        \leq& (1+o(1))k!\left(\log 3-\frac{3}{4}\log 2+o(1)\right)^{k}<k!\cdot 0.58^k.
    \end{align*}
    Here we used 
    \[
    \sum_{p}\frac{1}{p^2}=0.4522\ldots, \quad \sum_{p\leq x}\frac{1}{p}=\log\log x+M+o(1)
    \]
    where $M$ is a constant.   
    
    It remains to handle $A_1$. Using the Cauchy\textendash Schwarz inequality and the bound $x^{\sigma_{t,\chi-1/2}}\geq e^{2\eta}$, we find that
    \begin{equation}\label{prop thm 7 A_1}
        \begin{split}
            |A_1|\leq & \frac{1}{q-2}\left(B_1(\eta)+e^{-2\eta}\right)^{2k} \sqrt{\sum_{\chi\neq \chi_0} \left(\sigma_{t,\chi}-\frac{1}{2}\right)^{4k}x^{4k(\sigma_{t,\chi}-1/2)}} \\
            &\hspace{3cm}\cdot \sqrt{\sum_{\chi\neq \chi_0}\left(\int_{1/2}^\infty x^{1/2-\sigma}\left|\sum_{p<x^3}\frac{\Lambda_x(p)\log(xp)\chi(p)}{p^{\sigma+it}}\right|\md \sigma\right)^{4k}}
        \end{split}
    \end{equation}
    By Lemma~\ref{lemma 17 v2} applied with $r=v=4k$ (which is valid due to our assumption $\delta<2/(8k+3)=2/(2r+3)$),
    the first sum is
    \[
    \leq C(\eta,\delta,4k,4k)\frac{q-2}{(\log x)^{4k}}.
    \]
    For the other sum, an application of H\"{o}lder's inequality gives
    \begin{align*}
        \sum_{\chi\neq \chi_0}&\left(\int_{1/2}^\infty x^{1/2-\sigma}\left|\sum_{p<x^3}\frac{\Lambda_x(p)\log(xp)\chi(p)}{p^{\sigma+it}}\right|\md \sigma\right)^{4k} \\
        \leq& \sum_{\chi\neq \chi_0}\left(\int_{1/2}^\infty x^{1/2-\sigma}\md \sigma\right)^{4k-1} \int_{1/2}^\infty x^{1/2-\sigma} \left|\sum_{p<x^3}\frac{\Lambda_x(p)\log(xp)\chi(p)}{p^{\sigma+it}}\right|^{4k}\md \sigma\\
        =&\frac{1}{(\log x)^{4k-1}} \int_{1/2}^\infty x^{1/2-\sigma} \sum_{\chi\neq \chi_0}  \left|\sum_{p<x^3}\frac{\Lambda_x(p)\log(xp)\chi(p)}{p^{\sigma+it}}\right|^{4k}\md \sigma\\
        \leq &\frac{q-2}{(\log x)^{4k-1}} (1+o(1))(2k)!\int_{1/2}^\infty x^{1/2-\sigma} \left(\sum_{p<x^3}\frac{\Lambda_x(p)^2\log(xp)^2}{p^{2\sigma}}\right)^{2k}\md \sigma.
    \end{align*}
    By splitting this integral into two parts at $1/2+\Delta/\log x$, bounding each prime sum trivially by its value at the left endpoint of the interval, and then applying the prime number theorem, we see that the above line is
    \begin{align*}
        \leq& \frac{q-2}{(\log x)^{4k}} (1+o(1))(2k)!\left[(1-e^{-\Delta})\left(\sum_{p<x^3}\frac{\Lambda_x(p)^2\log(xp)^2}{p}\right)^{2k}+e^{-\Delta}\left(\sum_{p<x^3}\frac{\Lambda_x(p)^2\log(xp)^2}{p^{1+\frac{2\Delta}{\log x}}}\right)^{2k}\right]\\
        \leq& (1+o(1))^{2k}\frac{q-2}{(\log x)^{4k}}  (2k)!\left[(1-e^{-\Delta})\left(\int_1^{x^3} f(y)\md y\right)^{2k}+e^{-\Delta}\left(\int_1^{x^3} g_\Delta(y)\md y\right)^{2k}\right]
    \end{align*}
    for any $\Delta\geq 0$, where
    \[
    f(y):=\frac{\Lambda_x(y)^2\log(xy)^2}{y\log y}\quad \text{and} \quad g_\Delta(y):=\frac{\Lambda_x(y)^2\log(xy)^2}{y^{1+\frac{2\Delta}{\log x}}\log y}.
    \]
    We compute, by using the definition of $\Lambda_x(y)$ on corresponding intervals, 
    \[
    \frac{1}{\log^4 x}\int_1^{x^3} f_x(y)\md y=\left(\frac{17}{12}+\frac{6899}{1120} + \frac{3679}{3360} \right)<8.68.
    \]
    For the other integral, we use the upper bound $\Lambda_x(y)\leq \log y$ for simplicity and obtain
    \[
    \frac{1}{\log^4 x}\int_1^{x^3} g_\Delta(y)\md y<\frac{\left(2\Delta^{2}+4\Delta+3\right)e^{6\Delta}-192\Delta^{3}-80\Delta^{2}-22\Delta-3}{8\Delta^{4}e^{6\Delta}}.
    \]
    Returning to \eqref{prop thm 7 A_1}, we conclude that 
    \begin{align*}
        |A_1|<&  \left(B_1(\eta)+e^{-2\eta}\right)^{2k} \sqrt{\frac{C(\eta,\delta,4k,4k)}{(\log x)^{4k}}}\sqrt{(2k)!(\log x)^{4k} h(k)^{2k}}\\
        =& \left(B_1(\eta)+e^{-2\eta}\right)^{2k} \sqrt{C(\eta,\delta,4k,4k)\cdot (2k)!}\cdot h(k)^k
    \end{align*}
    where $h(k)$ was defined in \eqref{def h(k)}. Finally, collecting the estimates for each $|A_i|$ and inserting them into \eqref{Minkoski} completes the proof.
\end{proof}

\section{Proofs of Theorems~\ref{main theorem} and \ref{theorem: mean square S(t,chi)}}

We are now ready to prove our main theorems.

\subsection{Proof of Theorem~\ref{main theorem}}

Assume $|t|\leq 1$. Let $\eps=1/4$ (so that $|t|\leq q^{1/4-\eps}$), $0<\kappa<1/8$, $0<\delta<2/11$, and set $x= q^{\delta\kappa}$.
By the triangle inequality,
\begin{equation}\label{main theorem s1}
    \left|\sum_{\chi\neq \chi_0} S(t,\chi)\right|\le 
    \frac{1}{\pi}\sum_{p<x^3}\frac{1}{p^{1/2}}+\sum_{\chi\neq \chi_0} \left|S(t,\chi)-\frac{1}{\pi}\Im \sum_{p<x^3}\frac{\chi(p)}{p^{1/2+it}}\right|.
\end{equation}
The first sum on the right-hand side in \eqref{main theorem s1} is at most
\begin{equation*}
    \frac{1}{\pi}\int_0^{x^3} \frac{1}{u^{1/2}}\,du < \frac{2}{\pi} q^{3\kappa/11}<\frac{2}{\pi}q^{3/88}.
\end{equation*}
As for the second sum, we bound it using the Cauchy--Schwarz inequality. To this end, by Proposition~\ref{prop: selberg theorem 7} applied with $k=1$,
\begin{equation*}
    \sum_{\chi\neq \chi_0} \left|S(t,\chi)-\frac{1}{\pi}\Im \sum_{p<x^3}\frac{\chi(p)}{p^{1/2+it}}\right|^2 \le 
    (q-2)\cdot D(\eta,\delta,\kappa,1,1/4)
\end{equation*}
for $q\geq q_0(\delta,\kappa)$. Combining the last two estimates, we thus obtain
\begin{equation*}
    \frac{1}{q-2}\left|\sum_{\chi\neq \chi_0} S(t,\chi)\right| <  \frac{2}{\pi} q^{3/88-1}+ \sqrt{D(\eta,\delta,\kappa,1,1/4)}.
\end{equation*} 
To minimize this expression, we choose, after a numerical search, 
\begin{equation*}
    \eta=1.147,\qquad
    \delta = \frac{2}{11}\left(1-0.124\right)=0.159\ldots.
\end{equation*}
As we take $\kappa\to 1/8$ and $q\to \infty$, the resulting constant approaches $1074.8\ldots$, which justifies the claimed bound $C_0=1075$ when $q\geq q_0$.

\subsection{Proof of Theorem~\ref{theorem: mean square S(t,chi)}}
    
We choose the same parameters as in the proof of Theorem~\ref{main theorem}. In particular, $\eta=1.147$, $\delta=0.159\ldots$, $0<\kappa<1/8$. Since $|t|\leq \frac{2\pi\beta}{\log q}$ and $q$ is sufficiently large, we can assume $|t|\leq 1$. Write
\begin{multline*}
    \sum_{\chi\neq \chi_0} \wt{S}(t,\chi)^2
    =\sum_{\chi\neq \chi_0}\Bigg|S(t,\chi)-\frac{1}{\pi}\Im \sum_{p<x^3}\frac{\chi(p)}{p^{1/2+it}}+S(t,\ov{\chi})-\frac{1}{\pi}\Im \sum_{p<x^3}\frac{\ov{\chi(p)}}{p^{1/2+it}}\\
    +\frac{1}{\pi}\Im \sum_{p<x^3}\frac{\chi(p)+\ov{\chi(p)}}{p^{1/2+it}}\Bigg|^2.
\end{multline*}
Using Minkowski's inequality, we obtain
\begin{equation}\label{mean square S s1}
    \begin{split}
        \sqrt{\sum_{\chi\neq \chi_0} \wt{S}(t,\chi)^2}\leq &2\sqrt{\sum_{\chi\neq \chi_0}\Bigg|S(t,\chi)-\frac{1}{\pi}\Im \sum_{p<x^3}\frac{\chi(p)}{p^{1/2+it}}\Bigg|^2}\\
        &\hspace{2cm}+\sqrt{\sum_{\chi\neq \chi_0} \frac{1}{\pi^2}\Bigg| \sum_{p<x^3} \left(\chi(p)+\ov{\chi(p)}\right)\frac{\sin(t\log p)}{\sqrt{p}}\Bigg|^2}.
    \end{split}
\end{equation}
By Proposition~\ref{prop: selberg theorem 7} and the proof of Theorem~\ref{main theorem}, when we take $\kappa$ sufficiently close to $1/8$ and $q$ large, the first sum does not exceed
\[
D(\eta,\delta,\kappa,1,1/4) \cdot (q-2)<C_0^2 \cdot (q-2).
\]
We expand the other sum into 
\begin{align*}
    \frac{1}{\pi^2}\sum_{p_1,p_2<x^3}\sum_{\chi\neq \chi_0}\left(\chi(p_1p_2)+\chi(p_1p_2^{-1})+\chi(p_1^{-1}p_2)+\chi(p_1^{-1}p_2^{-1})\right)\frac{\sin(t\log p_1)\sin(t\log p_2)}{\sqrt{p_1p_2}}.
\end{align*}
Since $p_1p_2<x^6<q$, by orthogonality \eqref{orthogonality} this equals 
\begin{align*}
    =&(q-2)\frac{2}{\pi^2}\sum_{p<x^3}\frac{\sin(t\log p)^2}{p}+O\left(\left|\sum_{p<x^3}\frac{\sin(t\log p)}{\sqrt{p}}\right|^2\right)\\
    =&(q-2)\frac{2}{\pi^2}\left(\int_2^{x^3}\frac{\sin(t\log y)^2}{y\log y}\md y+O(t^2)\right).
\end{align*}
Here we used the prime number theorem as well as the bound $|\sin(t\log y)|\leq |t|\log y$. Since $|t|\leq \frac{2\pi\beta}{\log q}$ and $x^3=q^{3\delta\kappa}<q^{3\cdot 0.16/8}=q^{3/50}$, after a change of variable we deduce from \eqref{mean square S s1} that
\[
\frac{1}{q-2}\sum_{\chi\neq \chi_0} \wt{S}(t,\chi)^2< \left(2C_0+\frac{\sqrt{2}}{\pi}\sqrt{\int_0^{\frac{3\beta}{50}}\frac{\sin(2\pi y)^2}{y}\md y}\right)^2
\]
as $q\to \infty$, proving the first assertion. Via integration by parts, the preceding integral equals
\[
\frac{1}{2}\int_0^{\frac{3\beta}{50}}\frac{\md y}{y}+\frac{1}{2}\int_0^{\frac{3\beta}{50}} \frac{\cos(4\pi y)}{y}\md y =\frac{\log(\beta+1)}{2}+O(1),
\]
so the second assertion \eqref{E[S(t,chi)^2] asymptotics} holds as well.

\begin{remark}
    We restricted to low-lying height $|t|\leq \frac{2\pi\beta}{\log q}$ in the above argument, but considering the bound $|\sin(t\log y)|\leq \min\{1,|t|\log y\}$, we have in fact shown that for any $|t|\leq q^{1/4-\eps}$,
    \begin{multline*}
        \limsup_{q\to \infty}\bb{E}[\wt{S}(t,\chi)^2]\leq \frac{1}{\pi^2}\min\{\log\log q, \log(|t|\log q+3)\}\\
        +O\left(\min\left\{\sqrt{\log\log q}, \sqrt{\log(|t|\log q+3)}\right\}\right).
    \end{multline*}
\end{remark}

\printbibliography

\end{document}